\documentclass[12pt]{amsart} 
\usepackage{scrextend}
\usepackage{amsmath,amssymb,amsfonts,amsthm,amscd,amstext,amsxtra,amsopn,array,url,mathrsfs,enumerate,anysize,enumitem,mathabx} 
\marginsize{2.20cm}{2.20cm}{2.20cm}{2.20cm} 
\usepackage{times} 
\usepackage{amsrefs} 
\newenvironment{rezabib} 
{\bibdiv\biblist\setupbib} 
{\endbiblist\endbibdiv} 

\def\setupbib{\catcode`@=\active} 
\begingroup\lccode`~=`@ 
\lowercase{\endgroup\def~}#1#{\gatherkey{#1}} 
\def\gatherkey#1#2{\gatherkeyaux{#1}#2\gatherkeyaux} 
\def\gatherkeyaux#1#2,#3\gatherkeyaux{\bib{#2}{#1}{#3}} 
\usepackage{amsfonts}
\usepackage{amsmath}
\usepackage{amssymb}
\usepackage{amstext}
\usepackage{amsthm}
\usepackage{mleftright}

\usepackage{amssymb}
\usepackage{color}
\usepackage{hyperref}
\usepackage{tikz-cd}
\usepackage{multirow}
\usepackage{rotating}

\DeclareFontFamily{U}{wncy}{}
\DeclareFontShape{U}{wncy}{m}{n}{<->wncyr10}{}

\DeclareSymbolFont{yhlargesymbols}{OMX}{yhex}{m}{n} 
\DeclareMathAccent{\yhwidehat}{\mathord}{yhlargesymbols}{"62}

\newtheorem*{theorem*}{Theorem}

\newtheorem{theorem}{Theorem}[section]
\newtheorem{lemma}[theorem]{Lemma}
\newtheorem{proposition}[theorem]{Proposition}

\newtheorem{corollary}[theorem]{Corollary}
\theoremstyle{definition}
\newtheorem{definition}[theorem]{Definition}

\newtheorem{notes}[theorem]{Remarks}

\newtheorem{remark}[theorem]{Remark}

\numberwithin{equation}{section}
\usepackage[all]{xy}
\DeclareMathOperator{\Po}{Po}

\DeclareMathOperator{\Cl}{Cl}

\DeclareMathOperator{\Gal}{Gal}

\DeclareMathOperator{\Image}{Image}

\begin{document}


\title[Number Fields With Large P\'olya Groups]{Number Fields With Large P\'olya Groups} 

\author{Amir Akbary}
\address{Department of Mathematics and Computer Science, University of Lethbridge, Lethbridge, Alberta T1K 3M4, Canada}
\email{amir.akbary@uleth.ca}

\author{Abbas Maarefparvar} 
\address{Department of Mathematics and Computer Science, University of Lethbridge, Lethbridge, Alberta T1K 3M4, Canada} 
\email{abbas.maarefparvar@uleth.ca}


\thanks{Research of both authors is supported by NSERC. The second was supported by a PIMS postdoctoral fellowship and the University of Lethbridge during this research.}

\date{\today}

\begin{abstract} 
The P\'olya group $\Po(K)$ of a number field $K$ is the subgroup of the ideal class group $\Cl(K)$ of $K$ generated by the classes of all the products of the prime ideals of $K$ with the same norm.
Motivated by the classical ``one class in each genus problem'',
we prove general finiteness theorems for the number fields $K$ with a fixed P\'olya index $\left[\Cl(K):\Po(K)\right]$ in the families of Galois number fields, solvable CM-fields, and real quadratic fields of extended R-D type. We also give classification results for specific families. Most notably, we classify, unconditionally, all imaginary bi-quadratic and imaginary tri-quadratic fields with the P\'olya index one. Furthermore, 
we classify all real quadratic fields of extended R-D type (with possibly only one more field) with the P\'olya index one.
Also, under GRH, we give the complete list of 161 imaginary quadratic fields with the P\'olya index two. Finally, as a byproduct of our results, we extend, from narrow R-D types to the extended R-D types, Dohmae's classification of real quadratic fields of narrow R-D type whose narrow genus numbers equal their narrow class numbers.
\end{abstract} 
\maketitle

\noindent

\section{Introduction}

Let $\Cl(\Delta)$ denote the collection of equivalence classes of primitive binary quadratic forms of discriminant $\Delta$ and $h(\Delta)=\#\Cl(\Delta)$ be its corresponding class number. In Article 303 of \textit{Disquisitiones arithmeticae} \cite{Gauss}, Gauss 
conjectured that
for every positive integer $t$, there are only finitely many negative discriminants $\Delta$ such that $h(\Delta)=t$.
Heilbronn  \cite{Heilbronn} proved this conjecture. In other words, he showed that
\begin{equation*}
h(\Delta) \rightarrow \infty,  \quad \text{as} \quad -\Delta \rightarrow \infty.
\end{equation*}

Modifying Heilbronn's method, Chowla \cite{Chowla} proved the following result.

\begin{theorem}[Chowla] \label{theorem, Chowla}
Denote by $\omega(\Delta)$ the number of distinct prime factors of  discriminant $\Delta<0$. Then
\begin{equation*}
\frac{	h(\Delta)}{2^{\omega(\Delta)}} \rightarrow \infty,  \quad \text{as} \quad -\Delta \rightarrow \infty.
\end{equation*}
\end{theorem}
Gauss showed that $\Cl(\Delta)$ with an appropriate composition law forms a group. Let the \emph{genus} $g(\Delta)$ be the cardinality of the 2-torsion part of $\Cl(\Delta)$, i.e., $g(\Delta)=\#\Cl(\Delta)[2]$, where
\begin{equation*}
\Cl(\Delta)[2]=\{q\in \Cl(\Delta);~q^2\sim 1_\Delta\},
\end{equation*}
 in which $1_\Delta$ denotes the identity of the group $\Cl(\Delta)$. Gauss proved that $g(\Delta)=2^{\omega(\Delta)+\epsilon(\Delta)}$, where $\epsilon(\Delta)\in \{0, -1, -2\}$ (see \cite[Section 2.2]{Kani} for the exact formulas for $\epsilon(\Delta)$.) Therefore, Chowla's theorem implies that there are only finitely many negative discriminants $\Delta$ for which   
$h(\Delta)=t g(\Delta)$ for any fixed integer $t$. This statement was already conjectured, and a complete list of 65 such even discriminants $\Delta$ with $g(\Delta)=h(\Delta)$ was proposed by Gauss in the same Article 303. Considering the odd negative discriminants $\Delta$, Dickson \cite[pp. 88-89]{Dickson} gave a list of 101 discriminants $\Delta$ for which $g(\Delta)=h(\Delta)$.

Equivalently, the above results can be formulated using the terminology of quadratic number fields. For an imaginary quadratic field $K=\mathbb{Q}(\sqrt{-D})$, where $-D$ is a square-free negative integer, let the genus number $g_K$ be the cardinality of the $2$-torsion subgroup of the ideal class group $\Cl(K)$ of $K$. By letting $h_K=\#\Cl(K)$, the Dickson's list \cite[pp. 88-89]{Dickson} corresponds to $65$ square-free negative integers $-D$, given in Table \ref{table, classification imagianry quadratic}, for which $g_K=h_K$.
This is classically known as the list of imaginary quadratic fields with ``\emph{one class in each genus}'' (for fascinating facts regarding the numbers in this list and their relations with the prime producing \emph{idoneal numbers} of Euler see the comprehensive survey  \cite{Kani}.) 
Although the completeness of this list still is an open problem, under an assumption on the zeros of the Dirichlet $L$-function $L(s, \chi_{-D})$, 
Weinberger \cite[Thoeorem 1]{Weinberger} proved that this list is complete.

\begin{theorem}[{Weinberger}] \label{theorem, Weinberger}
There is at most one imaginary quadratic field $K=\mathbb{Q}(\sqrt{-D})$ with $D>1365$ such that $g_K=h_K$. Moreover, if $L(s, \chi_{-D})\neq 0$ for $D>2\cdot 10^{11}$ and $1-1/4\log{D} \leq s<1$, then Tabel \ref{table, classification imagianry quadratic} gives the complete list of imaginary quadratic fields $K$ for which $g_K=h_K$. 
\end{theorem}

\begin{table}[!h]
	\begin{center}
		\begin{tabular}{|c|p{10.1cm}|} 
			\hline
			$g_K=h_K$ & $D$   \\ [0.5ex] 
			\hline \hline
			$1$ & $1,2,3,7,11,19,43,67,163$ \\
			\hline 
			\multirow{2}{0.6em}{$2$} & $5,6,10,13,15,22,35,37,51,58,91,115,123,187,235,267,$ \\ 
			& $403,427$ \\
			\hline
			\multirow{2}{0.6em}{$4$} & $21,30,33,42,57,70,78,85,93,102,130,133,177,190,195,$ \\ 
			& $253,435,483,555,595,627,715,795,1435$ \\
			\hline
			\multirow{2}{0.6em}{$8$} & $105,165,210,273,330,345,357,385,462,1155,1995,3003,$ \\
			& $3315$ \\ \hline
			$16$ & $1365$ \\ \hline
		\end{tabular} \caption{The list of $65$  imaginary quadratic fields $K=\mathbb{Q}(\sqrt{-D})$ with $D \leq 1365$ for which $g_K=h_K$.
		} \label{table, classification imagianry quadratic}
	\end{center}
\end{table}

The genus theory of quadratic forms can be generalized in various ways to the general setting of number fields. A common generalization in the context of class field theory is given by Fr\"{o}hlich (see \cite{Ishida} for a detailed treatment.) 
  For a number field $K$, let $\Cl(K)$ (resp. $\Cl^+(K)$) be the class group (resp. narrow class group) of $K$, and let $h_K$ (resp. $h_K^+$) be the class number (resp. narrow class number) of $K$. Roughly speaking, the genus field (resp. narrow genus field) of $K$ is a specific abelian extension of $K$ whose Galois group is isomorphic to
 a quotient of $\Cl(K)$ (resp. a quotient of $\Cl^+(K)$). More formally, we have the following definition.

\begin{definition} 
\label{genus}
For a number field $K$, the \textit{genus field} $\Gamma_K$ (resp.  \textit{narrow genus field} $\Gamma_K^+$) of $K$ is the maximal abelian extension of $K$, which is the  compositum of an abelian number field $K_*$ with $K$ and is unramified at all places (resp. at all finite places) of $K$. The extension degrees $g_K=[\Gamma_K:K]$ and $g_K^+=[\Gamma_K^+:K]$ are called the \textit{genus number} and the \textit{narrow genus number} of $K$, respectively.
\end{definition}

The analogous problem to the classical one class in each genus for number fields has been studied previously {(see  \cite{CK}, \cite{Louboutin}, and \cite{Miyada})}. The family of imaginary quadratic fields is a special case of abelian CM-fields, and a finiteness result, as Theorem \ref{theorem, Chowla}, has been proven for such number fields.
Let $K$ be an \emph{abelian CM-field}, where by an abelian field, we mean a finite Galois extension of $\mathbb{Q}$ whose Galois group is abelian, and by a CM-field we mean a totally complex quadratic extension $K$ of a totally real field $K^+$ (Indeed, $K^+$ is the \emph{maximal totally real subfield} of $K$.)
 It can be shown that abelian CM-fields $K$ are the same as imaginary abelian fields (Let $K$ be an \emph{imaginary abelian field}, i.e., an abelian field with no real embedding. Let $\sigma$ be the usual complex conjugation. Then  $\sigma$ is an element of order two in the Galois group $G=\Gal(K/\mathbb{Q})$. Let $L$ be the fixed field of the subgroup $\langle \sigma \rangle$. Then $L=K \cap \mathbb{R}$, which implies that $L$ is a totally real Galois field, as $G$ is abelian. Since $[K:L]=2$, $K$ is an abelian CM-field. The converse is trivial.)
As stated before, an example of an abelian CM-field is an imaginary quadratic field. More generally, an imaginary multi-quadratic field, i.e., a field of the form $\mathbb{Q}(\sqrt{-n_1},\sqrt{-n_2},\dots,\sqrt{-n_t})$, for some distinct positive square-free integers $n_1,n_2,\dots,n_t$, is an abelian CM-field.

The following is proved in \cite[Theorem 1]{Louboutin}.
\begin{theorem} [{Louboutin}] \label{theorem, Louboutin0}
There are only finitely many abelian CM-fields $K$ such that $g_K=h_K$.
\end{theorem}

In fact, Loubotin proves a stronger assertion which gives Theorem \ref{theorem, Louboutin0} as a consequence of it. We denote by $d_K$ the absolute value of the discriminant of $K$. 
For an abelian CM-field $K$, let $h_K^-=h_K/h_{K^+}$, where $K^+$ is the maximal totally real subfield of $K$. 
The following generalization of Theorem \ref{theorem, Chowla} can be extracted from Louboutin's computations in the proof of Theorem \ref{theorem, Louboutin0} in \cite{Louboutin}.

\begin{theorem} \label{theorem, Louboutin00}
Let $K$ run through an infinite family $\mathcal{F}$ of abelian CM-fields. 
Then, for every $\epsilon >0$, we have
\begin{equation}
	\label{L-Chowla}
	\frac{(h_K^-)^{\epsilon}}{\prod_{p \mid d_K}e_p(K/\mathbb{Q})} \rightarrow \infty,  \quad \text{as} \quad d_K \rightarrow \infty
\end{equation} 
in the family $\mathcal{F}$, where $e_p(K/\mathbb{Q})$ denotes the ramification index of a prime $p$ in $K/\mathbb{Q}$. 
\end{theorem} Note that since $h_K^-\leq h_K$ and
$g_K\mid \prod_{p \mid d_K}e_p(K/\mathbb{Q})$ (see Theorem \ref{theorem, Leopoldt's result for the genus number}), then Theorem \ref{theorem, Louboutin00} for $\epsilon=1$ and family $\mathcal{F}$ of abelian CM-fields implies  Theorem \ref{theorem, Louboutin0}.
\medskip\par

In this paper, inspired by \cite{Emmelin}, we are interested in proving results similar to Theorems \ref{theorem, Louboutin0} and \ref{theorem, Louboutin00}  when $g_K$ is replaced with the order of a certain subgroup of $\Cl(K)$. 
Such a finiteness assertion, similar to Theorem \ref{theorem, Louboutin0},
is given by Emmelin \cite[Corollary 3.15]{Emmelin} for the P\'olya groups of abelian CM-fields.
For a number field $K$ and a prime power $q$, the ideal
$$\Pi_q(K)=:\prod_{\substack{N_{K/ \mathbb{Q}}(\mathfrak{p})=q}} \mathfrak{p}, $$
where $\mathfrak{p}$ denotes a prime ideal of $K$, is called an  
\emph{Ostrowski ideal}. If $K$ has no ideal with norm $q$, we set $\Pi_{q}(K)=\mathcal{O}_K$, the ring of integers of $K$.
The following is defined in \cite[Chapter II, $\S$3]{Cahen-Chabert's book}.
\begin{definition} 
\label{Polya}
	The \emph{P\'olya group} $\Po(K)$ of a number field $K$ is the subgroup of $\Cl(K)$ generated by the classes of all the Ostrowski ideals
$[\Pi_{q}(K)]$ of $K$.
\end{definition}

From the definition, it is evident that if $K=\mathbb{Q}(\sqrt{D})$ is a quadratic field, then $\Po(K)$ is a 2-group, hence it is a subgroup of $\Cl(K)[2]$. More precisely, for a quadratic field $K$, we have 
\begin{equation} \label{gK quadratic}
g_K=\#\Cl(K)[2]= \begin{cases}
2^{s_K-2},&{{\rm \textit{if}}~D>0,~N_{K/\mathbb{Q}}(\epsilon_K)=1,~{\rm \textit{and}}~D~{\rm \textit{has~ a~ prime ~divisor}}~p \equiv 3 \, (\mathrm{mod}\, 4),}\\
2^{s_K-1},&{{\rm \textit{otherwise}},}
\end{cases}
\end{equation}
and
\begin{equation}
\label{Hilbert}
\#\Po(K)=\begin{cases}
2^{s_K-2},&{{\rm \textit{if}}~D>0,~{\rm \textit{and}}~N_{K/\mathbb{Q}}(\epsilon_K)=1,}\\
2^{s_K-1},&{{\rm \textit{otherwise}},}
\end{cases} 
\end{equation}
where $s_K$ is the number of ramified primes in $K=\mathbb{Q} (\sqrt{D})$, and $\epsilon_K$ is the fundamental unit of $K$ when $D>0$ (see \cite[Corollary 9.8 and Theorem 9.10]{Lemmermeyer}.) Observe that $\#\Po(K)=g_K$ for imaginary quadratic fields $K$, and thus $\#\Po(K)$ can be considered as an extension of the classical genus number.
Also, $\#\Po(K)\mid g_K$ for all quadratic extensions $K/\mathbb{Q}$. However,  this is not true in general; for example, for the pure cubic field $K=\mathbb{Q}(\sqrt[3]{20})$, we have $\# \Po(K)=h_K=3$, while $g_K=1$ (see \cite[Example 3.10 ]{Leriche 2014}.)

By Theorem \ref{theorem, Zantema's exact sequence} below, we see that for Galois extensions $K/\mathbb{Q}$,
\begin{equation} 
\label{P-divide}
\# \Po(K) \mid {\prod_{p \mid d_K} e_{p}(K/\mathbb{Q})}.
\end{equation}
Thus, 
 \cite[Corollary 3.15]{Emmelin} can also be obtained from \eqref{L-Chowla}. In fact, using \eqref{L-Chowla} and \eqref{P-divide}, we  get the following finiteness theorem for \textit{P\'olya indices} $\left[\Cl(K):\Po(K)\right]$. 
\begin{corollary}
	For every positive integer $t$, there are only finitely many abelian CM-fields $K$ with $\left[\Cl(K): \Po(K)\right]=t$.
\end{corollary}

{\begin{notes}
(i) When $K/\mathbb{Q}$ is Galois, an ideal $\mathfrak{a}$ of $\mathcal{O}_K$ is called \emph{ambiguous} if $\mathfrak{a}^\sigma=\mathfrak{a}$ for all $\sigma\in {\rm Gal}(K/\mathbb{Q})$. A class $c\in \Cl(K)$ is called \emph{strongly ambiguous} 
if it contains an ambiguous ideal. It can be shown that the P\'{o}lya group of a Galois number field $K$ is the same as the group of strongly ambiguous ideal classes of $K$, see \cite[p. 10]{Zantema}.

(ii) We can find an example of a family of number fields for which $\Po(K)=\Cl(K)$ for infinitely many $K$ in that family. It is known that for $n \geq 3$ and an $S_n$-field $K$, 
i.e., a number field $K$ of degree $n$ whose normal closure has the Galois group isomorphic to the symmetric group $S_n$, we have $\Po(K)=\Cl(K)$ (see \cite[Theorem 4.1]{ChabertII} and \cite[Theorem 3.10]{Abbas}.) On the other hand, by \cite[Proposition 6.2]{Malle}, the number of $S_n$-fields of discriminant less than $x$ is bounded below by $cx^{1/n}$. Hence, in the family of number fields of degree $n$, there are infinitely many fields $K$ with $\Po(K)=\Cl(K)$.
\end{notes}
}

Our first result establishes versions of \eqref{L-Chowla} for families of CM-fields with solvable normal closures.

\begin{theorem} 
\label{theorem,CM}
Let $K$ run through an  {infinite} family $\mathcal{F}$ of CM-fields with solvable normal closures.
Write $d_K=d_{K^+}^2f$, where $K^+$ denotes the maximal totally real subfield of $K$.
\begin{itemize}
\item[(i)]  For every $\epsilon >0$, we have
\begin{equation*}
	\frac{(h_K^-)^{\epsilon}}{\prod_{p \mid f}e_p(K/\mathbb{Q})} \rightarrow \infty,  \quad \text{as} \quad d_K \rightarrow \infty	
\end{equation*}
in the family $\mathcal{F}$.
\item[(ii)] If $[K:\mathbb{Q}]$ remains bounded for those fields $K \in \mathcal{F}$ that are not abelian, then
\begin{equation*}
	\frac{(h_K^-)^{\epsilon}}{\prod_{p \mid d_K}e_p(K/\mathbb{Q})} \rightarrow \infty,  \quad \text{as} \quad d_K \rightarrow \infty
\end{equation*} 
in the family $\mathcal{F}$. 
\end{itemize}

\noindent
Moreover, the above results are effective for fields $K$ where $[K:\mathbb{Q}]\neq 2, 4$.
\end{theorem}

\begin{remark}
	The proof of Theorem \ref{theorem,CM} relies on the lower bounds for the class numbers $h_K^-$ given in \cite[Theorem 3.1]{Peng-Jie}.  Since, under the assumption of Artin's holomorphy conjecture,  such bounds hold for any family of CM-fields (see  \cite[Theorem 3.1]{Peng-Jie}), the assertions of Theorem \ref{theorem,CM} remain true for any family $\mathcal{F}$ of CM-fields under Artin's conjecture.
\end{remark}

As a direct consequence of part (ii) of the above theorem and \eqref{P-divide}, we have the following finiteness assertion for \textit{solvable CM-fields}.
\begin{corollary}
\label{cor-main}
{Let $K$ run through an infinite family $\mathcal{F}$ of solvable CM-fields such that $[K:\mathbb{Q}]$ remains bounded for non-abelian fields $K$ in this family. Then, for every positive integer $t$, there are only finitely many  $K \in \mathcal{F}$ with the P\'olya index $\left[\Cl(K): \Po(K)\right]=t$.}
 \end{corollary}

In light of Corollary \ref{cor-main}, it is natural to ask for classifying the complete list of abelian CM-fields $K$ or non-abelian solvable CM-fields $K$ of a given degree for which $\Po(K)=\Cl(K)$. In the context of Theorem  \ref{theorem, Louboutin0} (i.e., ``one class in each genus problem''), such a classification for abelian CM-fields is known. In \cite{CK},  
Chang and Kwon classified, unconditionally,  all non-quadratic abelian CM-fields $K$ with $g_K=h_K$. In contrast, such classification for abelian CM-fields $K$ with a given P\'olya index $\left[\Cl(K):\Po(K)\right]$ seems a challenging task, as computing $\# \Po(K)$ relies on computing the order of the cohomology group $H^1(\Gal(K/\mathbb{Q}),U_K)$ (see Theorem \ref{theorem, Zantema's exact sequence}.) Yet, by studying these cohomology groups for some specific
CM-fields, and
employing some of the classification results of \cite{CK}  
we establish the following classifications.
\begin{theorem}
\label{theorem,classification}
In the family of abelian CM-fields, the following assertions hold.
\begin{itemize}
	\item[(i)] There are $65$ imaginary quadratic fields $K=\mathbb{Q}(\sqrt{-D})$, with possibly only one more value $D$,  such that $\Po(K)=\Cl(K)$.  The list of such fields coincides with the list given in Table \ref{table, classification imagianry quadratic}.
	\item[(ii)] There are exactly $57$ imaginary bi-quadratic fields $K=\mathbb{Q}(\sqrt{-m}, \sqrt{-n})$ such that $\Po(K)=\Cl(K)$. The complete list of such fields is given in Table \ref{tab, bi-quadratic Po=h}.
	\item[(iii)] There are exactly $17$ imaginary tri-quadratic fields $K=\mathbb{Q}(\sqrt{-m_1}, \sqrt{-m_2}, \sqrt{-m_3})$ such that $\Po(K)=\Cl(K)$. The complete list of such fields is given in Table  \ref{tab, tri-quadratic Miyada}.
	\item[(iv)] There are exactly $77$ imaginary non-quadratic cyclic fields $K$ such that $\Po(K)=\Cl(K)$. The list of such fields coincides with the list of all imaginary non-quadratic cyclic fields $K$ with $g_K=h_K$ given in \cite[Table 1]{CK}. 
	\end{itemize}
\end{theorem}

In \cite[Artcile 303]{Gauss} Gauss conjectures that
$$\Delta=-4d~ {\rm where}~ d=14, 17, 20, 32, 34, 36, 39, 46, 49, 52, 55, 63, 64, 73, 82, 97, 100, 142, 148, 193$$
is the complete list of even negative discriminants $\Delta$ with $(h(\Delta), g(\Delta))=(4, 2)$, i.e., the discriminants of genus two with \textit{two classes in each genus}. We now know that the above list is complete, as the class number four problem is solved for negative discriminants (see \cite{Arno}).
Inspired by Gauss's list, 
we conditionally classify imaginary quadratic fields with the P\'{o}lya index two.
\begin{theorem}
\label{theorem, list imaginary [Cl(K):Po(K)]=2}
Under the assumption of the Generalized Riemann Hypothesis (GRH), there are exactly 161 imaginary quadratic fields $K$ with the P\'olya index $\left[ \Cl(K): \Po(K)\right]=2$ (equivalently $h_K=2g_K$). The complete list of such fields is given in Table \ref{tab, imaginiary quadratic t=2}. 
\end{theorem}
As noted before, the finiteness results as Theorem \ref{theorem,CM} (replacing $h_K^{-}$ with $h_K$) for families of fields that are not CM may not hold. For example, for any $S_n$ field $K$ we have $\Po(K)=\Cl(K)$. We next, as a direct consequence of the Brauer-Siegel theorem, show that finiteness results similar to Theorem \ref{theorem,CM} (replacing $h_K^{-}$ with $h_K$) hold for 
families of Galois number fields whose degrees and regulators grow modestly with respect to their discriminants.

{\begin{theorem} \label{theorem, Rk=O(dK)}
Let $K$ run through an infinite family $\mathcal{F}$ of Galois number fields.
Suppose that 
\begin{equation}
\label{BS-condition2}
\lim_{d_K { \rightarrow} \infty} \frac{ {[K:\mathbb{Q}]}}{\log {d_K}}=0,
\end{equation}
and 
\begin{equation}
\label{limsup2}
\limsup_{d_K\rightarrow \infty} \frac{\log{R_K}}{\log{d_K}}<\frac{1}{2}
\end{equation}
in the family $\mathcal{F}$, where $R_K$ denotes the regulator of $K$.
Then, for every $\epsilon >0$, we have
\begin{equation}
\label{hK2}
	\frac{(h_K)^{\epsilon}}{\prod_{p \mid d_K}e_p(K/\mathbb{Q})} \rightarrow \infty,  \quad \text{as} \quad d_K \rightarrow \infty
\end{equation} 
in the family $\mathcal{F}$.
\end{theorem}

\begin{remark}
By the Brauer-Siegel theorem, for a family $\mathcal{F}$ of Galois number fields satisfying \eqref{BS-condition2}, we have
\begin{equation*}
\limsup_{d_K\rightarrow \infty} \frac{\log{R_K}}{\log{d_K}}\leq\frac{1}{2}. 
\end{equation*}
Thus \eqref{limsup2} is equivalent to not having $1/2$ as a limit point of the set 
$\{\log{R_K}/\log{d_K}:~K\in \mathcal{F}\}$.
\end{remark}

It is known that \eqref{BS-condition2} holds for any infinite family $\mathcal{F}$ of abelian fields (see \cite[Lemma 2]{Louboutin}.) Thus, we get the following corollary.

\begin{corollary} \label{corollary, 1/2 is not limit abelian}
		The assertion \eqref{hK2} holds for a family $\mathcal{F}$ of abelian fields as long as  $1/2$ is not a limit point of the set 
	$\{\log{R_K}/\log{d_K}:~K\in \mathcal{F}\}.$
\end{corollary}
}
As a concrete example of a family satisfying the conditions of Corollary \ref{corollary, 1/2 is not limit abelian}, we consider the family of real quadratic fields of {extended R-D type}.

{\begin{definition} 
A real quadratic field $K=\mathbb{Q}(\sqrt{D})$ (or simply square-free $D$) is said to be of \emph{Richaud-Degert type}, \emph{R-D type} for short, if $D=\ell^2+r \neq 5$, where $ - \ell < r \leq \ell$, and $r \mid 4 \ell$. If $r \in \{ \pm 1, \pm 4\}$, then $D$ is said to be of \emph{narrow R-D type}. If we remove the condition  $ - \ell < r \leq \ell$ and include $D=5$, 
then $D$ is called \emph{extended R-D type}.
\end{definition}
}
Since there are infinitely many square-free integers in the form $n^2+1$  (see \cite{Nagel}), then the set of extended R-D types is infinite.
For real quadratic fields $K$ of extended R-D type, one has $R_K < \log (3 d_K)$ (see Lemma \ref{lemma, log R_K less than log dK}.) Hence, the assertion of Corollary \ref{corollary, 1/2 is not limit abelian} holds for this family, and we get the following finiteness result.

\begin{corollary} 
\label{theorem,RD}
For a fixed positive integer $t$, there are only finitely many real quadratic fields $K$ of extended R-D type with  $\left[\Cl(K): \Po(K)\right]=t$. 
\end{corollary}

To illustrate an instance of the above corollary, we classify, with possibly one exception, all real quadratic fields of extended R-D type whose P\'olya groups are equal to their ideal class groups.

\begin{theorem}
\label{theorem, finiteness Extended R-D type}
	There are $269$ real quadratic fields $K=\mathbb{Q}(\sqrt{D})$ of extended R-D type, with possibly only one more
	value $D \geq 6.3 \times 10^{16}$, such that $\Po(K)=\Cl(K)$. These fields are listed in Table \ref{tab,E.R-D}.
\end{theorem}

We next observe a close connection between the P\'olya indices of real quadratic fields and their genus indices. For a number field $K$, as before, let $h_K^+$ and $g_K^+$ be, respectively,  its narrow class number and its narrow genus number. 
In {Corollary} \ref{corollary, Po(K),gK and gK+ real quadratic}, we show, for a real quadratic field $K$,
\begin{equation}
\label{index relation}
\frac{\# \Po(K)}{h_K}=\frac{g_K^+}{h_K^+}.
\end{equation}
Hence, for the quadratic fields (imaginary or real), the P\'{o}lya index one problem has the same solution as the classical ``one class in each genus problem''.

In \cite{Kazuhiro}, Dohmae classified all real quadratic fields $K$ of narrow R-D type (with possibly only one more field) whose narrow genus numbers $g_K^+$ are equal to their narrow class numbers $h_K^+$.

\begin{theorem} [{Dohmae}]  
	\label{theorem, Kazuhiro0}
	Let $D$ be of narrow R-D type. Then the following assertions hold.
	\begin{itemize} 
		\item[(i)] There are 69 values for $D$, given in Table \ref{tab, Kazuhiro}, such that $D \leq 1.31 \times 10^{16}$ and for $K=\mathbb{Q}(\sqrt{D})$ we have $g_K^+=h_K^+$. 	
		\item[(ii)]  There is at most one $D$ such that $D> 1.31 \times 10^{16}$ and for $K=\mathbb{Q}(\sqrt{D})$ we have $g_K^+=h_K^+$.
	\end{itemize}
\end{theorem}

\begin{table}[!h]
	\begin{center}
		\begin{tabular}{| c|p{8.2cm}|} 
			\hline
			 $g^+_K=h^+_K$ & $D$   \\ [0.5ex] 
			\hline \hline
			 $1$ & $2,17, 29, 37, 53, 101, 173, 197, 293, 677$ \\
			\hline
		$2$ & $3, 21, 77, 437$ \\   \hline
			\multirow{2}{0.5em}{$2$} & $10, 26, 65, 85, 122, 143, 362, 365, 485, 533, 629, $ \\ 
			 & $965, 1157, 1685, 1853, 2117, 2813, 3365$ \\ \hline
			 	$4$ & $15,35,165, 285,357, 1085, 2397$ \\  \hline
			
		\multirow{2}{0.5em}{$4$} & $170,290, 530, 957, 962, 1370, 2405, 3485, 9605,$ \\ 
			 & $ 10205, 14885, 16133, 20165$ \\ \hline
			 $8$ & $195,255, 483, 1295, 1365, 2805, 4485, 7917, 8645$ \\ \hline
			$8$ & $2210, 5330, 32045, 58565, 77285$ \\  \hline 
		 $16$ & $1155, 3135, 26565$ \\ \hline
		\end{tabular} \caption{The list of real quadratic fields $K=\mathbb{Q}(\sqrt{D})$ of narrow R-D type with $D \leq 1.31 \times 10 ^{16}$ for which $g^+_{K}=h^+_{K}$. 
		} \label{tab, Kazuhiro}
	\end{center}
\end{table}

Our Theorem \ref{theorem, finiteness Extended R-D type} together with \eqref{index relation} provides an extension of Theorem \ref{theorem, Kazuhiro0} to real quadratic fields of extended R-D type.

\begin{corollary}
There are $269$ real quadratic fields $K=\mathbb{Q}(\sqrt{D})$ of extended R-D type, with possibly only one more
	value  $D \geq 6.3 \times 10^{16}$, such that $g_K^+=h_K^+$ (equivalently, $\#\Po(K)=h_K$.) The list of such fields is given in Table \ref{tab,E.R-D}.
\end{corollary}

Finally, by using the results of Theorem \ref{theorem, finiteness Extended R-D type} and solving a P\'{o}lya index two problem, we classify, with possibly only one more, all real quadratic fields of extended R-D type whose genus numbers are equal to their class numbers. 
	
\begin{theorem}
\label{theorem, E-R-D real quadratic gK=hK}	
	There are $298$ real quadratic fields $K=\mathbb{Q}(\sqrt{D})$ of extended R-D type, with possibly only one more
	value  $D$, such that $g_K=h_K$. Among them, $269$ fields are those $K$ listed in Table \ref{tab,E.R-D} with $D < 6.3 \times 10^{16}$ for which  $g_K^+=h_K^+$. The remaining $29$ fields $K$ are the ones listed in Table \ref{tab,E.R-D2} with $D < 4.3 \times 10^{18}$ for which $ g_K^+ \neq h_K^+$. 
\end{theorem}

Our proofs employ various lower bounds for the class numbers of number fields. More precisely, the proof of Theorem \ref{theorem,CM} uses the lower bounds for the relative class numbers of abelian CM-fields given by Louboutin in \cite[Proof of Theorem 3.1]{Louboutin}, and the ones provided in Proposition \ref{proposition, bounds for hk-}, derived from Wong \cite{Peng-Jie} and Murty \cite{Murty}. The lower bound employed in Theorem \ref{theorem, Rk=O(dK)} is a direct consequence of the Brauer-Siegel theorem. The proof of Theorem \ref{theorem, list imaginary [Cl(K):Po(K)]=2} uses Ihara's explicit lower bound (conditional on GRH), as given in Theorem \ref{theorem, Ihara's upper bound}. The classification results for extended R-D types  (Theorems \ref{theorem, finiteness Extended R-D type} and \ref{theorem, E-R-D real quadratic gK=hK}) follow from the explicit lower bound (with one possible exception) given by Mollin and Williams \cite[p. 423]{MW}, derived from Tautuzawa's lower bound \cite{Tatuzawa} for the values of the Dirichlet $L$-functions at $1$. 
The proof of the classification result of Theorem \ref{theorem,classification} exploits the relations between the order of the P\'{o}lya groups and the genus numbers of abelian CM-fields. More precisely, for the fields $K$ in the families  considered in Theorem \ref{theorem,classification} it is shown that $\#\Po(K)\mid g_K$, and since $g_K \leq h_K$, the classification of the P\'olya index one problem in these cases reduces to the ``one class in each genus problem'' for abelian CM-fields which already obtained by Chang and Kwon \cite{CK}.
Verifying the relation $\#\Po(K)\mid g_K$ for abelian CM-fields $K$  in Theorem \ref{theorem,classification} uses a fundamental result of Zantema \cite{Zantema} (see Theorem \ref{theorem, Zantema's exact sequence}) and studying the cohomology groups $H^1(\Gal(K/\mathbb{Q}),U_K)$. 

The structure of the paper is as follows. In Section \ref{section, genus group} we summarize results on genus fields and P\'{o}lya groups and the relation between the orders of the P\'{o}lya groups and the genus numbers. The proof of the relation $\#\Po(K) \mid g_K$ for imaginary tri-quadratic fields $K$ (Proposition \ref{proposition, Po(K) and gK for K imaginary tri-quadratic}) is a highlight of this section. The general finiteness theorems (Theorems \ref{theorem,CM} and \ref{theorem, Rk=O(dK)}) are proved in Section \ref{three}. The proof of the unconditional classification of P\'{o}lya index one fields for imaginary bi-quadratic and tri-quadratic fields (Theorem \ref{theorem,classification}) is given in Section \ref{four}. Finally, the details of the conditional classification of imaginary quadratic fields of P\'{o}lya index 2 (Theorem \ref{theorem, list imaginary [Cl(K):Po(K)]=2}) and the classification results for real quadratic fields of extended R-D types (Theorems \ref{theorem, finiteness Extended R-D type} and \ref{theorem, E-R-D real quadratic gK=hK}) are provided in Section \ref{five}. We tabulated our classification results in Section \ref{six}.
\medskip\par

{\noindent \bf{Notation.}} The following notations will be used throughout this article.
For a number field $K$, the notations  $[K:\mathbb{Q}]$, $\Cl(K)$, $h_K$, $h_K^+$, $g_K$, $g_K^+$, $d_K$, $\mathcal{O}_K$, $U_K$, $\mu_K$, and $R_K$ denote the degree, ideal class group, class number, narrow class number, genus number, narrow genus number, absolute value of discriminant, ring of integers, group of units, group of roots of unity, and regulator of $K$, respectively. By $U_K^+$ we denote the group of totally positive units of $K$, and $	U_K^2=\{u^2 \, : \, u \in U_K\}$. For a real quadratic field $K$, we denote the fundamental unit of $K$ by $\epsilon_K$.
For a CM-field $K$ with the maximal totally real subfield $K^+$, the number $h_K^-=h_K/h_{K^+}$ denotes the relative class number of $K$. Also, we denote the Hasse unit index $\left[U_K:\mu_K U_{K^+}\right]$ by $W_K$. 
We use  $N_{K/\mathbb{Q}}$ to denote both the ideal norm and the element norm map from $K$ to $\mathbb{Q}$. For a Galois extension $K/\mathbb{Q}$,  we denote the ramification index of $p$ in $K/\mathbb{Q}$ by $e_p(K/\mathbb{Q})$. The notations $H^i(\Gal(K/\mathbb{Q}),U_K)$ (for an integer $i \geq 0$) and  $\yhwidehat{H}^{0}(\Gal(K/\mathbb{Q}),U_{K})$ denote
 the $i^{\text{th}}$ cohomology group and the $0^{\text{th}}$ Tate cohomology group of $\Gal(K/\mathbb{Q})$  with the coefficients in $U_K$, respectively. 
  We denote the cardinality of a finite set $S$ by $\# S$, and the number of divisors of an integer $n$ by $\tau(n)$.
\medskip\par
\noindent
\textbf{Acknowledgment.} The authors would like to  thank Kumar Murty and Peng-Jie Wong for their comments on an earlier version of this article. 

\section{Genus fields and P\'olya groups} \label{section, genus group}
In this section, we give explicit relations between the genus numbers and the orders of the P\'olya groups of some Galois number fields. 
For genus numbers, we follow the notation and terminology of \cite{Ishida}.
Recall that $g_K^+=[\Gamma_K^+: K]$ is the narrow genus number of $K$, where $\Gamma_K^+$ is the narrow genus field of $K$ (see Definition \ref{genus}).
For abelian number fields, Leopoldt \cite{Leopoldt} gave the following genus number formula.

\begin{theorem} [{Leopoldt}] \label{theorem, Leopoldt's result for the genus number}
	Let $K$ be an abelian number field. Then
	\begin{equation} \label{equation, genus number formula}
		g_K^{+}=\frac{\prod_{p|d_K} e_{p}(K/\mathbb{Q})}{[K:\mathbb{Q}]},
	\end{equation}
	where $e_{p}(K/\mathbb{Q})$ denotes the ramification index of a prime $p$ in $K/\mathbb{Q}$. 
\end{theorem}

Recall that the P\'olya group $\Po(K)$ of $K$ is the subgroup of $\Cl(K)$ formed by the classes of all Ostrowski ideals of $K$ (see Definition \ref{Polya}.)  For P\'{o}lya groups we use the notation and terminology of \cite[Chapter II, $\S$3]{Cahen-Chabert's book}. The following fundamental result for P\'{o}lya groups of Galois number fields is proved in 
\cite[ Section 3]{Zantema}.
\begin{theorem}[{Zantema}]  \label{theorem, Zantema's exact sequence}
Let $K/\mathbb{Q}$ be a finite Galois extension with Galois group $G$. Then the  sequence 
	\begin{equation} \label{equation, Zantema's exact sequence}
		0 \longrightarrow H^1(G,U_K) \longrightarrow \bigoplus _{p | d_K} \frac{\mathbb{Z}}{e_{p} (K/\mathbb{Q}) \mathbb{Z}}   \longrightarrow \Po(K) \longrightarrow 0
	\end{equation}
	is exact, where $H^1(G,U_K)$ is the first cohomology group of $G$ with coefficients in the unit group $U_K$ of K, and $e_{p}(K/\mathbb{Q})$ denotes the ramification index of a prime $p$ in $K/\mathbb{Q}$.
\end{theorem}
As a direct corollary of \eqref{equation, Zantema's exact sequence}, we have
\begin{equation}
\label{Zantema}
\#\Po(K)=\frac{\prod_{p|d_K} e_{p}(K/\mathbb{Q})}{\# H^1(G,U_K) }.
\end{equation}

Combining \eqref{equation, genus number formula} with  \eqref{Zantema}, we can explicitly relate the order of the P\'olya groups of abelian number fields to their narrow genus numbers.  

\begin{proposition} \label{proposition, relation between PoK and gK}
	For an abelian number field $K$, we have
	\begin{equation} \label{equation, relation between PoK and gK}
		\#\Po(K)=\frac{[K:\mathbb{Q}]}{\# H^1(\Gal(K/\mathbb{Q}),U_K)}\cdot g_K^{+}.
	\end{equation}
\end{proposition}
Proposition \ref{proposition, relation between PoK and gK} gives us a key tool to investigate abelian number fields whose P\'olya groups are equal to their class groups. In particular, as applications of Formula \eqref{equation, relation between PoK and gK}, we find that the order of Po(K) divides the narrow genus number of K in some specific families of abelian number fields $K$. 

\begin{proposition} \label{proposition, Po(K) and gK for K cyclic}
	Let $K$ be a cyclic number field. Then
		\begin{equation} \label{equation, Po(K) and gK+ for K cyclic}
		\# \Po(K) = \left\{
		\begin{array}{ll}
			g_K^+/2, \,  & \,  \text{if $K$ is real and $N_{K/\mathbb{Q}}(U_K)=\{+1\}$},\\
			& \\
			g_K^+, \, & \, \text{otherwise.} \\	
		\end{array}
		\right.	
	\end{equation}
In addition, in the case that $K$ is a real cyclic field, we have
\begin{equation} \label{equation, UK+ Uk2}
	 \frac{\# \Po(K)}{h_K} = \left\{
	\begin{array}{ll}
		\frac{[U_K^+: U_K^2]}{2}\cdot \frac{g_K^+}{h_K^+}, \,  & \, \text{if} \, N_{K/\mathbb{Q}}(U_K)=\{+1\}, \\
		& \\
		{[U_K^+: U_K^2]}\cdot \frac{g_K^+}{h_K^+}, \, & \, \text{otherwise,} \\	
	\end{array}
	\right.	
		\end{equation}
	where $h_K^+$ denotes the narrow class number of $K$, $U_K^+$ denotes the totally positive units of $K$, and
	\begin{equation*}
		U_K^2=\{u^2 \, : \, u \in U_K\}.
	\end{equation*}	
\end{proposition}
\begin{proof}
For a cyclic number field $K$, it is known that 
$\#H^1(\Gal(K/\mathbb{Q}),U_K)=2[K:\mathbb{Q}]$ if $K$ is real and $N_{K/\mathbb{Q}}(U_K)=\{+1\}$, and is $[K:\mathbb{Q}]$ otherwise (see \cite[Section 3]{Zantema}.) Thus, \eqref{equation, Po(K) and gK+ for K cyclic} follows from \eqref{equation, relation between PoK and gK}. To obtain \eqref{equation, UK+ Uk2}, we use  \eqref{equation, Po(K) and gK+ for K cyclic} and 
\begin{equation*}
	h_K^+=h_K.\left[U_K^+:U_K^2\right],
\end{equation*}
see \cite[Theorem 1.7]{milneCFT}.     
\end{proof}
As a consequence of the above proposition, we obtain the following result for real quadratic fields. 

\begin{corollary} \label{corollary, Po(K),gK and gK+ real quadratic}       
If $K=\mathbb{Q}(\sqrt{D})$ is a real quadratic field, for a  square-free integer $D>0$, then

\begin{equation} \label{equation, Po(K) and gK for K real quadratic}
	\# \Po(K)=
	\left\{
	\begin{array}{ll}
		g_K/2, &  \text{if $N_{K/\mathbb{Q}}(\epsilon_K)=+1$ and  $D$ has no prime divisor $p \equiv 3 \, (\mathrm{mod}\, 4)$},\\
		g_K,  &  \text{otherwise},\\
	\end{array}
	\right.
\end{equation}
where $\epsilon_K$ denotes the fundamental unit of $K$. Moreover,
\begin{equation} \label{equation, Po(K), hK, hK+ for K real quadratic}
	\frac{\# \Po(K)}{h_K}=\frac{g_K^+}{h_K^+}.
\end{equation}
\end{corollary}

\begin{proof}
We have
	\begin{equation} \label{equation, gK gK+ real quadratic}
		g_K^+=
		\left\{
		\begin{array}{ll}
			2g_K, &  \textit{if $D$ has a prime divisor $p \equiv 3 \, (\mathrm{mod}\, 4)$},\\
			g_K,  &  \textit{otherwise},\\
		\end{array}
		\right.
	\end{equation}
	see \cite[Proposition 3.11]{Leriche 2014}. 
If $D$ has a prime divisor  $p \equiv 3 \, (\mathrm{mod}\, 4)$, then $N_{K/\mathbb{Q}}(\epsilon_{K})=+1$. Now by employing  \eqref{equation, gK gK+ real quadratic} and
Proposition \ref{proposition, Po(K) and gK for K cyclic} we get the relation \eqref{equation, Po(K) and gK for K real quadratic} (Note that \eqref{equation, Po(K) and gK for K real quadratic} can be also obtained from \eqref{gK quadratic} and \eqref{Hilbert}.) In addition,
 using Dirichlet's Unit Theorem, we have
\begin{equation} \label{equation, UK:UK+=4}
	4=\left[U_K:U_K^2\right]=\left[U_K:U_K^+\right]\cdot \left[U_K^+:U_K^2\right].
\end{equation}
Now if $N_{K/\mathbb{Q}}(\epsilon_K)=-1$, then $\left[U_K:U_K^+\right]=4$, since $-1,\epsilon_K \not \in U_K \backslash U_K^+$. But, if $N_{K/\mathbb{Q}}(\epsilon_K)=+1$, then $\epsilon_K \in U_K^+$ which implies that $\left[U_K:U_K^+\right]=2$. Consequently, by \eqref{equation, UK:UK+=4} we obtain 
		\begin{equation} \label{equation, UK UK+}
		\left[U_K^+:U_K^2\right] =
		\begin{cases}
			2, \,  & \textit{if} \, \, N_{K/\mathbb{Q}}(\epsilon_K)=+1, \\
			1, \,  & \textit{if}  \, \, N_{K/\mathbb{Q}}(\epsilon_K)=-1.
		\end{cases}
	\end{equation}
	From \eqref{equation, UK+ Uk2} and  \eqref{equation, UK UK+} we get \eqref{equation, Po(K), hK, hK+ for K real quadratic}.
\end{proof}
We show that the divisibility of $g_K$ by $\# \Po(K)$ also holds for imaginary bi-quadratic and tri-quadratic fields $K$. 
\begin{proposition} \label{proposition, Po(K) and gK for K imaginary bi-quadratic}
Let $K$ be an imaginary bi-quadratic field. Then
\begin{equation*} 
	\# \Po(K) = \left\{
	\begin{array}{ll}
		g_K/2, \,  & \, \text{if} \, \, W_K=N_{K^+/\mathbb{Q}}(\epsilon_{K^+})=1, \, \\
		& \\
		g_K, \, & \, \text{otherwise,} \\	
	\end{array}
	\right.
\end{equation*}
where $K^+$ denotes the real quadratic subfield of $K$, with the fundamental unit $\epsilon_{K^+}$, and $			W_K=\left[U_K:\mu_K\cdot U_{K^+}\right]$
denotes the Hasse unit index of $K$ in which  $\mu_K$ is the group of roots of unity in $K$.
\end{proposition}

\begin{proof}
	First, observe that if $K=\mathbb{Q}(\sqrt{-1},\sqrt{-2})$, then $\# \Po(K)=g_K=h_K=1$ (Note that in this case, $K^+=\mathbb{Q}(\sqrt{2})$ has units of negative norm.) So, let $K \neq \mathbb{Q}(\sqrt{-1},\sqrt{-2})$. Then,
from \cite[Lemma 4.3]{Zantema}, we have
\begin{equation*}
	H^1(\Gal(K/\mathbb{Q}),U_K) \simeq \left\{
	\begin{array}{ll}
		\mathbb{Z}/2\mathbb{Z} \oplus \mathbb{Z}/2\mathbb{Z} \oplus \mathbb{Z}/2\mathbb{Z}, & \textit{if} \, \, W_K=N_{K^+/\mathbb{Q}}(\epsilon_{K^+})=1,\\
		& \\
		\mathbb{Z}/2\mathbb{Z} \oplus \mathbb{Z}/2\mathbb{Z} , & \textit{otherwise.} \\	
	\end{array}
	\right.
\end{equation*}
Now the assertion follows from the above isomorphism and Proposition \ref{proposition, relation between PoK and gK}.
\end{proof}

We next consider the imaginary tri-quadratic fields.

\begin{proposition} \label{proposition, Po(K) and gK for K imaginary tri-quadratic}
If $K$ is an imaginary tri-quadratic field, then $\# \Po(K) \mid g_K$.
\end{proposition}

\begin{proof}
Let $K=\mathbb{Q}(\sqrt{-m_1},\sqrt{-m_2},\sqrt{-m_3})$, where $m_1,m_2,m_3$ are three distinct square-free positive integers. We show that 
\begin{equation*} 
	8 \mid \# H^1(\Gal(K/\mathbb{Q}),U_K).
\end{equation*}
Then the assertion follows from Proposition \ref{proposition, relation between PoK and gK}. 

Let $K^+=\mathbb{Q}(\sqrt{m_1m_2},\sqrt{m_1m_3})$. Following Setzer's terminology in \cite{Bennett},  for $i\in\{0,1,2,3\}$, we say $K^+$ is of the type $M_i$  if $i$ of the quadratic subfields of $K^+$ have units of negative norm. For a quadratic subfield $k_j$ of $K^+$, denote by $\epsilon_j$ the fundamental unit of $k_j$. 
 We have two cases:
 	
\smallskip
\textit{(i)}
	At most one of the subfields $k_j$'s has units with negative norm. Hence, $K^+$ is of the type $M_0$ or $M_1$. In this case, by \cite[Table I]{Bennett}, we find that
	\begin{equation} \label{equation, 8 divides H^1K+}
		8 \mid \# H^1(\Gal(K^+/\mathbb{Q}),U_{K^+}).
	\end{equation}
Since the cohomology group $H^1(\Gal(K^+/\mathbb{Q}),U_{K^+})$ is embedded in $H^1(\Gal(K/\mathbb{Q}),U_K)$ via the {inflation map}, then \eqref{equation, 8 divides H^1K+} yields the claim.

\smallskip	
\textit{(ii)}
At least two of the subfields $k_j$'s, say $k_1$ and $k_2$, have units with negative norm, i.e., $N_{k_1/\mathbb{Q}}(\epsilon_1)=N_{k_2/\mathbb{Q}}(\epsilon_2)=-1$. Hence,  $K^+$ is of the type $M_2$ or $M_3$. From \cite[Table I]{Bennett}, we observe that either
	\begin{equation} \label{equation, H1=8 and Unit index is 1}
		\# H^1(\Gal(K^+/\mathbb{Q}),U_{K^+})=8 \quad \text{and} \quad \left[U_{K^+}:U_{k_1}U_{k_2}U_{k_3}\right]=1,
	\end{equation}
	or
	\begin{equation} \label{equation, H1=4 and Unit index is 2}
		\# H^1(\Gal(K^+/\mathbb{Q}),U_{K^+})=4 \quad \text{and} \quad \left[U_{K^+}:U_{k_1}U_{k_2}U_{k_3}\right]=2.
	\end{equation}
	If \eqref{equation, H1=8 and Unit index is 1} holds, then similar to the case \textit{(i)}, the inflation map implies
	\begin{equation*}
		8 \mid \# H^1(\Gal(K/\mathbb{Q}),U_{K}).
	\end{equation*}
	Hence, we assume that \eqref{equation, H1=4 and Unit index is 2} holds. Since  $N_{k_1/\mathbb{Q}}(\epsilon_1)=-1$,  the norm map 
	\begin{equation*}
		N_{k_1/\mathbb{Q}}:U_{k_1} \rightarrow U_{\mathbb{Q}}=\{\pm 1\}
	\end{equation*}
is surjective. Consequently, as $k_1/\mathbb{Q}$ is cyclic, by \cite[Proposition 1.6.12]{Neukirch}, we have
	\begin{equation}\label{trivial}
		H^2(\Gal(k_1/\mathbb{Q}),U_{k_1}) \simeq \yhwidehat{H}^0(\Gal(k_1/\mathbb{Q}),U_{k_1})=\frac{U_{\mathbb{Q}}}{N_{k_1/\mathbb{Q}}(U_{k_1})}=\{0\},
	\end{equation}
	where 
		$\yhwidehat{H}^{0}(\Gal(k_1/\mathbb{Q}),U_{k_1})$ denotes the $0^{\text{th}}$ Tate cohomology group of $\Gal(k_1/\mathbb{Q})$ with coefficient in $U_{k_1}$. In this case,
	for the subgroup $\Gal(K/k_1)$ of $\Gal(K/\mathbb{Q})$, we have the following {inflation-restriction} exact sequence 
	\begin{equation} \label{equation, Inf-Res for K/k1}
		0 \rightarrow H^1(\Gal(k_1/\mathbb{Q}),U_{k_1}) \xrightarrow{\text{inf}} H^1(\Gal(K/\mathbb{Q}),U_K) \xrightarrow{\text{res}} H^1(\Gal(K/k_1),U_K) \rightarrow 	\underbrace{H^2(\Gal(k_1/\mathbb{Q}),U_{k_1})}_{0}.
	\end{equation} 
	
	We first compute the order of $H^1(\Gal(k_1/\mathbb{Q}),U_{k_1})$. By \eqref{trivial},  for the cyclic extension $k_1/\mathbb{Q}$,
	the Herbrand quotient $Q(\Gal(k_1/\mathbb{Q}),U_{k_1})$ is given by
	\begin{equation} \label{equation, Herbrand for k_1}
		Q(\Gal(k_1/\mathbb{Q}),U_{k_1})=\frac{\# \yhwidehat{H}^0(\Gal(k_1/\mathbb{Q}),U_{k_1})}{\# H^1(\Gal(k_1/\mathbb{Q}),U_{k_1})}=\frac{1}{\# H^1(\Gal(k_1/\mathbb{Q}),U_{k_1})}.	
	\end{equation}
	On the other hand, by \cite[Proposition 5.10]{NChildress},
	\begin{equation} \label{equation, Herbrand2 for k_1}
		Q(\Gal(k_1/\mathbb{Q}),U_{k_1})=\frac{2^{s_{k_1/\mathbb{Q}}}}{[k_1:\mathbb{Q}]},
	\end{equation}
	where $s_{k_1/\mathbb{Q}}$ denotes the number of infinite places of $\mathbb{Q}$ ramified in $k_1$.
	Since $k_1$ is real, 
	from \eqref{equation, Herbrand for k_1} and \eqref{equation, Herbrand2 for k_1}, we get
	\begin{equation} \label{equation, order of H1k_1 is 2}
		\# H^1(\Gal(k_1/\mathbb{Q}),U_{k_1})=2.
	\end{equation}
	
	Next, we show that   
	\begin{equation}
	\label{four1}
	4 \mid \# H^1(\Gal(K/k_1),U_K),
	\end{equation}
	which together with \eqref{equation, Inf-Res for K/k1} and \eqref{equation, order of H1k_1 is 2} imply 
	\begin{equation*}
		8 \mid \# H^1(\Gal(K/\mathbb{Q}),U_K),
	\end{equation*}
	as desired. 
	
	We consider two subcases:

\smallskip
\textit{(ii-a).} 
			Let $K^+$ be of the type $M_2$, i.e., $N_{k_3/\mathbb{Q}}(\epsilon_3)=+1$. In this case, to prove \eqref{four1},
		we show that, in fact,
	\begin{equation*}
			4 \mid \# H^1(\Gal(K^+/k_1),U_{K^+})
	\end{equation*}
(Note that $H^1(\Gal(K^+/k_1),U_{K^+})$ is embedded in $H^1(\Gal(K/k_1),U_K)$ via the inflation map.) For the cyclic extension $K^+/k_1$, the Herbrand quotient $Q(\Gal(K^+/k_1),U_{K^+})$ is given by
\begin{equation*} 
	Q(\Gal(K^+/k_1),U_{K^+})=\frac{2^{s_{K^+/k_1}}}{[K^+:k_1]},
\end{equation*}
where $s_{K^+/k_1}$ denotes the number of infinite places of $k_1$ ramified in $K^+$. Hence
\begin{equation} \label{equation, Herbrand Quotient for K+/k1}
	Q(\Gal(K^+/k_1),U_{K^+})=\frac{1}{2}=\frac{\# \yhwidehat{H}^0(\Gal(K^+/k_1),U_{K^+})}{\# H^1(\Gal(K^+/k_1),U_{K^+})}.
\end{equation}
To deduce that the order of $ H^1(\Gal(K^+/k_1),U_{K^+})$ is divisible by $4$, we show that
\begin{equation*}
	 \# \yhwidehat{H}^0(\Gal(K^+/k_1),U_{K^+}) >1.
\end{equation*}
Since
\begin{equation*}
\# \yhwidehat{H}^0(\Gal(K^+/k_1),U_{K^+})=\left[U_{k_1}: N_{K^+/k_1}(U_{K^+})\right],
\end{equation*}
it is enough to show that the norm map
		\begin{equation*}
		N_{K^+/k_1}:U_{K^+} \rightarrow U_{k_1}
	\end{equation*} 
is not surjective.		
			Following \cite[p. 161]{Bennett}, define
		{\small 
			\begin{equation} \label{equation, the set C}
				\mathcal{C}=\{(u_1,u_2,u_3) \in U_{k_1} \times U_{k_2} \times U_{k_3} : N_{k_1/\mathbb{Q}}(u_1)=N_{k_2/\mathbb{Q}}(u_2)=N_{k_3/\mathbb{Q}}(u_3)=\text{sign}(u_1u_2u_3)\}.
		\end{equation}}
		Then, for the function
		\begin{align*} \label{equation, the norm function phi}
			\varphi : U_{K^+} & \rightarrow U_{k_1} \times U_{k_2} \times U_{k_3} \\ 
			u & \mapsto \left( N_{K^+/k_1}(u),N_{K^+/k_2}(u), N_{K^+/k_3}(u) \right), \nonumber
		\end{align*}
		we have
		\begin{equation} \label{equation, Im(phi) subset C}
			\Image(\varphi) \subseteq \mathcal{C},
		\end{equation}
		 see \cite[p. 161]{Bennett}.
		 Now if the norm map $N_{K^+/k_1}$ is surjective, then for the element $u_0 \in U_{K^+}$ with $N_{K^+/k_1}(u_0)=\epsilon_1$ we have
		\begin{equation*}
			\varphi(u_0)=\left( \epsilon_1, N_{K^+/k_2}(u_0),N_{K^+/k_3}(u_0)\right) \in \mathcal{C}.
		\end{equation*}
		Thus, by \eqref{equation, the set C}, we get a unit of $k_3$, say $N_{K^+/k_3}(u_0)$, with negative norm which contradicts the assumption $N_{k_3/\mathbb{Q}}(\epsilon_3)=+1$. Hence, if $K^+$ is of type $M_2$, then the norm map 
		\begin{equation*}
			N_{K^+/k_1}:U_{K^+} \rightarrow U_{k_1}
		\end{equation*} 
can not be surjective, as claimed. Therefore, by the above argument, \eqref{four1} holds.
		
\smallskip		
\textit{(ii-b).} 
		Let $K^+$ be of the type $M_3$, i.e., $N_{k_3/\mathbb{Q}}(\epsilon_3)=-1$. If	the norm map
		\begin{equation*}
			N_{K^+/k_1}:U_{K^+} \rightarrow U_{k_1}
		\end{equation*}
	is not surjective, similar to the case \textit{(ii-a)}, we have 
$4 \mid \# H^1(\Gal(K^+/k_1),U_{K^+})$ which yields \eqref{four1}.
So, assume that the norm map $N_{K^+/k_1}$ is surjective. Then, 
\begin{equation*}
	H^2(\Gal(K^+/k_1),U_{K^+}) \simeq 	\yhwidehat{H}^0(\Gal(K^+/k_1),U_{K^+})=\{0\}.
\end{equation*} 
In this case, from \eqref{equation, Herbrand Quotient for K+/k1}, we get
 \begin{equation}
 \label{two}
 	\# H^1(\Gal(K^+/k_1),U_{K^+})=2,
 \end{equation}
and for the field extensions $k_1 \subseteq K^+ \subseteq K$, we get the following inflation-restriction exact sequence
{\small 
\begin{equation} \label{equation, Inf-Res M3}
0 \rightarrow H^1(\Gal(K^+/k_1),U_{ K^+}) \xrightarrow{\text{inf}} H^1(\Gal(K/k_1),U_K) \xrightarrow{\text{res}} H^1(\Gal(K/K^+),U_K) \rightarrow \underbrace{H^2(\Gal(K^+/k_1),U_{K^+})}_{0}.
\end{equation}} 
Now we show that 
\begin{equation} \label{equation, H1 K/K+=2}
\# H^1(\Gal(K/K^+),U_K) =2.
\end{equation}
Then from \eqref{two} and \eqref{equation, Inf-Res M3}, we obtain
$ \# H^1(\Gal(K/k_1),U_{K})=4$ which implies \eqref{four1}.

Since $K/K^+$ is cyclic, the Herbrand Quotient $Q(\Gal(K/K^+),U_K)$ is given by
\begin{equation} \label{equation, Herbrand quotient for K/K+}
Q(\Gal(K/K^+),U_K)=\frac{2^{s_{K/K^+}}}{[K:K^+]}=\frac{\# \yhwidehat{H}^0(\Gal(K/K^+),U_K)}{\# H^1(\Gal(K/K^+),U_K)},
\end{equation}
where $s_{K/K^+}$ denotes the number of infinite places of $K^+$ ramified in $K$. Therefore, from \eqref{equation, Herbrand quotient for K/K+} we get
\begin{equation} \label{equation, H1 1/8}
\# H^1(\Gal(K/K^+),U_K)=\frac{1}{8} \cdot \# \yhwidehat{H}^0(\Gal(K/K^+),U_K)=\frac{1}{8} \cdot [U_{K^+}:N_{K/K^+}(U_K)].
\end{equation}
We next show that
\begin{equation*}
[U_{K^+}:N_{K/K^+}(U_K)]=16,
\end{equation*}
to obtain \eqref{equation, H1 K/K+=2} from \eqref{equation, H1 1/8}. Observe that if $U_K=\mu_K\cdot U_{K^+}$, then by  Dirichlet's Unit Theorem, we have
\begin{equation*}
	[U_{K^+}:N_{K/K^+}(U_K)]=[U_{K^+}:N_{K/K^+}(\mu_K\cdot U_{K^+})]=[U_{K^+}: U_{K^+}^2]=16,
	\end{equation*}
	as we desire
(Note that since $K/K^+$ is a quadratic extension and $K^+$ is totally real, we have $N_{K/K^+}(\mu_K)=\{+1\}$.) Hence, it is enough to show that $W_K=\left[U_K : \mu_K \cdot U_{K^+}\right]=1$.
		If $\mathbb{Q}(\sqrt{-1},\sqrt{-2}) \subseteq K$,  by \cite[Theorem 2]{Hirabayashi}, we have $W_K=1$. Also,
		if $\mathbb{Q}(\sqrt{-1},\sqrt{-2}) \not \subseteq K$, then by Lemma \ref{lemma, root-epsilon} below, again we have  $W_K=1$ (Recall that we assumed \eqref{equation, H1=4 and Unit index is 2} holds, in particular the unit index $\left[U_{K^+}: U_{k_1} U_{k_2} U_{k_3}\right]$ is 2.)
Thus, $[U_{K^+}:N_{K/K^+}(U_K)]=16$.  
From here by \eqref{two}, \eqref{equation, Inf-Res M3}, and  \eqref{equation, H1 1/8} we get
\begin{equation*}
	\# H^1(\Gal(K/k_1),U_K)=4,
\end{equation*}
which implies that \eqref{four1} holds in this case, too.
\end{proof}

\begin{lemma} \label{lemma, root-epsilon}
With the above notations, suppose that $N_{k_i/\mathbb{Q}}(\epsilon_{k_i})=-1$ for every $i=1,2,3$, and $\left[U_{K^+}:U_{k_1}U_{k_2}U_{k_3}\right]=2$. If $\mathbb{Q}(\sqrt{-1},\sqrt{-2}) \not \subseteq K$, then  $W_K=\left[U_K : \mu_K U_{K^+}\right]$ equals to $1$.
\end{lemma}

\begin{proof}
Since each $\epsilon_i$ has a negative norm, by \cite[Satz 1]{Kubota}, a system of fundamental units for $U_{K^+}$ is  given by either  $\{\epsilon_{k_1},\epsilon_{k_2},\epsilon_{k_3}\}$, or  $\{\sqrt{\epsilon_{k_1} \epsilon_{k_2} \epsilon_{k_3}}, \epsilon_{k_2},\epsilon_{k_3}\}$. Since
\begin{equation*}
	 \left[U_{K^+}:U_{k_1}U_{k_2}U_{k_3}\right] = 2,
\end{equation*}
the latter case holds.
In particular, $\sqrt{\epsilon_{k_1} \epsilon_{k_2} \epsilon_{k_3}} \in K^+$. By \cite[Proof of Theorem 1, Case (4)]{Hirabayashi}, if $W_K=2$ then $K^+(\sqrt{\epsilon_{k_1} \epsilon_{k_2} \epsilon_{k_3}})$ is a quadratic extension of $K^+$, hence $\sqrt{\epsilon_{k_1} \epsilon_{k_2} \epsilon_{k_3}} \not \in K^+$. Therefore, we must have $W_K=1$.
\end{proof}

\section{Proofs of finiteness theorems}\label{three}
The following lower bounds,  given by Wong \cite{Peng-Jie}, on the relative class numbers $h_K^-$ of CM-fields $K$ with solvable normal closures,  play a crucial role in our proofs.
\begin{theorem}\label{theorem,Peng-Jie}
(i) Let $K$ be a CM-field with solvable normal closure and with the maximal totally real subfield $K^+$. Write $d_K=d_{K^+}^2 f$, and let $n=[K^+:\mathbb{Q}]$. Then there are effective constants $c_3 >0$ and $c_4 >1$ such that
\begin{equation} \label{equation, 1st bound Peng-Ji}
h_K^- \geq \frac{c_3 c_4^n f^{\left(\frac{1}{2}-\frac{1}{2n} \right)}}{n (2n)^{e(2n)} \delta (2n)} ,
\end{equation}
where
\begin{equation}
\label{delta(n)}
e(n)=\max_{p^{\alpha} | n} \alpha, \quad \text{and} \quad 	
\delta (n)=\left(e(n)+1 \right)^2 3^{\frac{1}{3}} 12^{e(n)-1}.
\end{equation}
Moreover, for any $\kappa \in \left[\frac{1}{8 \log {d_K}}, 1\right]$, there exists an effective constant $c(\kappa) >0$, depending only on $\kappa$, such that
\begin{equation} \label{equation, 2nd bound Peng-Ji}
	h_K^- \geq \frac{c_3 c(\kappa)^n}{n (2n)^{e(2n)} \delta (2n)} d_{K^+}^{(\frac{1}{2}-\frac{1}{n}-\kappa)} f^{(\frac{1}{2}-\frac{1}{2n})}.
\end{equation}
(ii) The lower bound \eqref{equation, 2nd bound Peng-Ji} can be improved as
\begin{equation} \label{equation, 3rd bound Peng-Ji}
	h_K^- \geq \frac{\tilde{c}_3 c(\kappa)^n}{n (2n)^{e(2n)} \delta (2n)} d_{K^+}^{(\frac{1}{2}-\frac{1}{2n}-\kappa)} f^{(\frac{1}{2}-\frac{1}{4n})},
\end{equation}
for an ineffective constant $\tilde{c}_3 >0$.
\end{theorem}
\begin{proof}
For (i) see \cite[Theorem 3.1]{Peng-Jie}, and for  (ii) See   \cite[Proof of Theorem 1.2]{Peng-Jie}.
\end{proof}

\begin{proposition} \label{proposition, bounds for hk-}
In the family of CM-fields, the following assertions hold.
\begin{itemize}
	\item[(i)] There exists a constant $C>0$ such that for any abelian CM-field $K$ with $d_K >C$, we have
\begin{equation*}
h_{K}^{-} \geq d_K^{\frac{1}{16}}.
\end{equation*} 
Moreover, if $[K:\mathbb{Q}] \geq 6$, then the constant $C$ is effective.
	\item[(ii)] There exists an ineffective constant $c>0$ such that for any CM-field $K $with $[K:\mathbb{Q}]=4$, we have
	\begin{equation*}
	h_K^- \geq c d_K^{\frac{1}{12}}.
	\end{equation*}	
\item[(iii)] If $K$ is a CM-field with solvable normal closure, $[K:\mathbb{Q}]=2n \geq 6$,  and $d_K=d_{K^+}^2f$, then there is an effective {constant} $c(n)$, where $\lim_{n \rightarrow \infty} c(n) =\infty$, such that
\begin{equation*}
h_K^- \geq c(n) f^{\frac{1}{3}}.
\end{equation*}
	
\item[(iv)] If $K$ is a CM-field with solvable normal closure such that $[K:\mathbb{Q}]=2n \geq 6$, then there is an effective {constant} $b(n) >0$, such that 
\begin{equation*}
h_K^- \geq b(n) d_K^{\frac{1}{24}}.
\end{equation*}
\end{itemize}
\end{proposition}

\begin{proof}
(i) See \cite[Proof of Theorem 1] {Louboutin}.

(ii) If $[K:\mathbb{Q}]=4$, then $[\tilde{K}:\mathbb{Q}] \leq 24$, where $\tilde{K}$ denotes the normal closure of $K$ over $\mathbb{Q}$. Hence $\tilde{K}/\mathbb{Q}$ is a solvable extension. Write $d_K=d_{K^+}^2f$. Then, using the lower bound \eqref{equation, 3rd bound Peng-Ji} for $n=2$ and $\kappa=\frac{1}{12}$, there is an ineffective constant $c>0$ such that, for $d_K \geq e^{\frac{3}{2}}$,
\begin{align*}
h_K^- \geq 
c d_{K^+}^{\frac{1}{6}} f^{\frac{3}{8}} \geq c d_K^{\frac{1}{12}}.
\end{align*}

(iii) By the lower bound \eqref{equation, 1st bound Peng-Ji} given in Theorem \ref{theorem,Peng-Jie}, we have
\begin{equation*}
	h_K^- 
	\geq c(n) f^{\left(\frac{1}{2}-\frac{1}{6}\right)}=c(n)f^{\frac{1}{3}},  
\end{equation*}
where
\begin{equation*}
c(n):=\frac{c_3 c_4^n}{n (2n)^{e(2n)} \delta (2n)}. 
\end{equation*}
 Note that $c_4>1$ and by \eqref{delta(n)}, $e(n)\leq \log{n}/\log{2}$. Hence, $\lim_{n\rightarrow \infty} c(n)=\infty$.

(iv) Let $\kappa \in \left[\frac{1}{8 \log {d_K}}, 1\right]$ be given.  
Employing $d_K=d_{K^+}^2 f$ in \eqref{equation, 2nd bound Peng-Ji} we get
\begin{equation} \label{equation, hK- geq kappa/20}
	h_K^- \geq  \frac{c_3 c(\kappa)^n}{n(2n)^{e(2n)} \delta(2n)} d_K^{(\frac{1}{4}-\frac{1}{2n}-\frac{\kappa}{2})} f^{(\frac{1}{4}+\frac{\kappa}{2})}.
\end{equation}
For $\kappa=\frac{1}{12}$, from \eqref{equation, hK- geq kappa/20}, we get
\begin{equation*}
	h_K^- \geq b(n) d_K^{\frac{1}{24}}, \quad \text{for} \, \, d_K \geq e^{3/2},
\end{equation*}
where 
\begin{equation} \label{equation, b(n)}
b(n)=\frac{c_3 c(1/12)^n}{n(2n)^{e(2n)} \delta(2n)}.
\end{equation}

\end{proof}

We are now ready to prove our first result.

\begin{proof}[Proof of Theorem \ref{theorem,CM}.] 
(i) Let $M, \epsilon >0$ and for $K \in \mathcal{F}$ set
\begin{equation}
\label{set1}
\frac{(h_K^-)^{\epsilon}}{\prod_{p \mid f}e_p(K/\mathbb{Q})} < M.
\end{equation}

We consider cases. If $[K:\mathbb{Q}]=4$,  then by
 employing part (ii) of Proposition \ref{proposition, bounds for hk-} in \eqref{set1}, and $d_K=d_{K^+}^2 f$, we get
\begin{equation}
\label{set2}
 \prod_{p \mid d_{K^+}} p^{\frac{ m_p \epsilon}{16}} 
\prod_{p \mid f} \frac{p^{\frac{m_p \epsilon}{16}} }{e_p(K/\mathbb{Q})} \leq \frac{(h_K^-)^{\epsilon}}{\prod_{p \mid f} e_p(K/\mathbb{Q})} <M,
\end{equation}
where $d_K=\prod_{p\mid d_K} p^{m_p}$. By \cite[Chapter III, Proposition 8]{Lang} we have 
$e_p(K/\mathbb{Q})\leq 2m_p$. Hence, from \eqref{set2} we have
\begin{equation}
\label{set3}
f(d_K):= \prod_{p \mid d_{K^+}}
p^{\frac{ m_p \epsilon}{16}} \prod_{p \mid f} \frac{p^{\frac{m_p \epsilon}{16}} }{2m_p} \leq \frac{(h_K^-)^{\epsilon}}{\prod_{p \mid f} e_p(K/\mathbb{Q})} <M.
\end{equation}
By \cite[Theorem 316]{Hardy}, $ f(d_K) \rightarrow \infty$, as $d_K\rightarrow \infty$.
Hence, from \eqref{set3} we conclude that the number of $K \in \mathcal{F}$ with 
 $[K:\mathbb{Q}]=4$ satisfying \eqref{set1} is finite. 

Using a similar argument and employing part (i) of Proposition \ref{proposition, bounds for hk-}, we have that the number of $K \in \mathcal{F}$ with $[K:\mathbb{Q}]=2$
satisfying \eqref{set1} is finite (Note that in this case, $d_{K^+}=1$ as $K^+=\mathbb{Q}$.) 

Now assume that $2n=[K:\mathbb{Q}] \geq 6$, $K \in \mathcal{F}$, and \eqref{set1} holds.
Then, from part (iii) of Proposition \ref{proposition, bounds for hk-}, we have
\begin{equation*}
c(n)^{\epsilon} \prod_{p | f} \frac{p^{\frac{m_p \epsilon}{3}}}{e_p(K/\mathbb{Q})} \leq \frac{(h_K^-)^{\epsilon}}{\prod_{p \mid f} e_p(K/\mathbb{Q})} < M.
\end{equation*}
Since $c(n) \rightarrow \infty$, as $n \rightarrow \infty$, and $e_p(K/\mathbb{Q})\leq 2m_p$ implies the product over $f$ is bounded below by some positive constant, we conclude that the degrees $n$ for such $K$'s are bounded. Now from part (iv) of Proposition \ref{proposition, bounds for hk-} and \eqref{set1}, we get
\begin{equation*}
b(n)^{\epsilon} \prod_{p \mid d_{K^+}} p^{\frac{ m_p \epsilon}{{24}}} \prod_{p | f} \frac{p^{\frac{m_p \epsilon}{24}}}{e_p(K/\mathbb{Q})} \leq \frac{(h_K^-)^{\epsilon}}{\prod_{p \mid f} e_p(K/\mathbb{Q})} < M,
\end{equation*}
 where $b(n)$ is the function given by \eqref{equation, b(n)}.
Applying $e_p(K/\mathbb{Q})\leq 2m_p$ in  the above inequality yields
\begin{equation}\label{set4}
j(d_K):=  \prod_{p \mid d_{K^+}} p^{\frac{ m_p \epsilon}{{ 24}}} \prod_{p | f} \frac{p^{\frac{m_p \epsilon}{24}}}{2 m_p} \leq \frac{(h_K^-)^{\epsilon}}{b(n)^{\epsilon}\prod_{p \mid f} e_p(K/\mathbb{Q})} < \frac{M}{b(n)^{\epsilon}}.
\end{equation}

As above, since by \cite[Theorem 316]{Hardy}, $j(d_K) \rightarrow \infty$, as $d_K\rightarrow \infty$, and  $b(n)$ is bounded, as $n$ is bounded, then \eqref{set4} shows that, the number of $K \in \mathcal{F}$ with $[K:\mathbb{Q}]\geq 6$ satisfying \eqref{set1} is finite. 

Hence, for given $M,\epsilon >0$, there are only finitely many $K \in \mathcal{F}$ satisfying \eqref{set1}. Thus, the claimed assertion holds.

(ii) Similar to part (i), for $M, \epsilon >0$ and  $K \in \mathcal{F}$ set
\begin{equation}
\label{set5}
\frac{(h_K^-)^{\epsilon}}{\prod_{p \mid d_K}e_p(K/\mathbb{Q})} < M
\end{equation}
and consider cases.

Suppose that there are infinitely many CM-fields $K$ in the family $\mathcal{F}$ such that $K/\mathbb{Q}$ is not abelian and $ 2n=[K:\mathbb{Q}]$ remains bounded as $d_K \rightarrow \infty$. 
Note that for such fields $K \in \mathcal{F}$, we must have $[K:\mathbb{Q}]=2n \geq 4$. Now, for $2n \geq 6$, by employing part (iv) of 
Proposition \ref{proposition, bounds for hk-} and 
$e_p(K/\mathbb{Q})\leq 2m_p$, from \eqref{set5} we get
\begin{equation}
\label{set6}
r(d_K):= \prod_{p \mid d_K} \frac{p^{\frac{m_p \epsilon}{24}} }{2m_p} \leq \frac{(h_K^-)^{\epsilon}}{b(n)^{\epsilon}\prod_{p \mid d_K} e_p(K/\mathbb{Q})} <\frac{M}{b(n)^\epsilon}.
\end{equation}
Since, by \cite[Theorem 316]{Hardy}, $r(d_K) \rightarrow \infty$ as $d_K\rightarrow \infty$, and  $b(n)$ is bounded, as $n$ is bounded, then \eqref{set6} shows that, the number of non-abelian CM-fields $K \in \mathcal{F}$ satisfying \eqref{set5} is finite. 

An identical argument which uses part (i) (respectively part (ii)) of 
Proposition \ref{proposition, bounds for hk-},  instead of part (iv), establishes the finiteness of abelian  (respectively quartic non-Galois) fields in the family $\mathcal{F}$ satisfying \eqref{set5}.
\end{proof}

\begin{proof}[Proof of Theorem \ref{theorem, Rk=O(dK)}.]
Since \eqref{BS-condition2} holds, by Brauer-Siegel's theorem,
\begin{equation}
\label{Brauer-Siegel2}
\lim_{d_K \rightarrow \infty} \frac{\log (h_K R_K)}{\log{d_K}}=\frac{1}{2}.		
\end{equation}
Let $\eta>0$ be such that 
\begin{equation}
\label{limsup3}
\limsup_{d_K\rightarrow \infty} \frac{\log{R_K}}{\log{d_K}}<\frac{1}{2}-2\eta.
\end{equation}
(Note that by \eqref{limsup2} such $\eta$ exists.)
Hence, from \eqref{Brauer-Siegel2} and \eqref{limsup3}, there exists a $C_{\eta}>0$ such that for $d_K>C_{\eta}$ we have
\begin{equation*}
\frac{1}{2}-\eta <\frac{\log {(h_K R_K)}}{\log{d_K}}<\frac{\log {h_K}}{\log{d_K}} +\frac{1}{2}-2\eta.
\end{equation*}
Thus, for $d_K > C_{\eta}$, we have $h_K > {d_K}^{\eta}$. Now, let 
	$d_K=\prod_{p \mid d_K} p ^{m_p}.$
By \cite[Chapter III, Proposition 8]{Lang} we have 
$e_p(K/\mathbb{Q})\leq 2m_p$.
Then, for a given $\epsilon>0$ and $K$ in the family $\mathcal{F}$  with $d_K>C_{\eta}$, we have
\begin{equation*}
\frac{(h_K)^{\epsilon}}{\prod_{p \mid d_K} e_p(K/\mathbb{Q})} > \prod_{p \mid d_K} \frac{(p^{m_p})^{\eta \epsilon}}{e_p(K/\mathbb{Q})} \geq  \prod_{p \mid d_K} \frac{(p^{m_p})^{\eta \epsilon}}{2 m_p}.
\end{equation*}
By \cite[Theorem 316]{Hardy} the right-hand side of the above inequality approaches $\infty$ as $d_K \rightarrow \infty$. Thus, the result holds.
\end{proof}

\section{Proof of Theorem \ref{theorem,classification}}\label{four}
\begin{proof} 
(i) By part (i) of Proposition \ref{proposition, Po(K) and gK for K cyclic}, we have $\# \Po(K)=g_K$ for an imaginary quadratic field $K$. Now the assertion follows from Theorem \ref{theorem, Weinberger}.

(ii) By Proposition \ref{proposition, Po(K) and gK for K imaginary bi-quadratic}, we have
\begin{equation}  \label{equation, order of Po divides gK bi-quadratic}
	\# \Po(K) = \left\{
	\begin{array}{ll}
		g_K/2, \,  & \, \textit{if} \, \, W_K=N_{K^+/\mathbb{Q}}(\epsilon_{K^+})=1, \, \\
		& \\
		g_K, \, & \, \textit{otherwise,} \\	
	\end{array}
	\right.
\end{equation} 
which implies $\#\Po(K) \leq g_K$. Since $ g_K \leq h_K$,  the number fields in \cite[Table III]{CK} (also in \cite[Table 2]{Miyada}),
consisting of all  imaginary bi-quadratic fields $K$ with $g_K=h_K$, form an \emph{initial list} of all such $K$ with $\Po(K)=\Cl(K)$.  Now it suffices to exclude those fields $K$ from \cite[Table III]{CK}  for which $\#\Po(K)=g_K/2$, and for the remaining cases we will have $\#\Po(K)=h_K$. By \eqref{equation, order of Po divides gK bi-quadratic}, the desired list is formed by all the fields $K$ in \cite[Table III]{CK}  for which $W_K \neq N_{K^+/\mathbb{Q}}(\epsilon_{K^+})$. To check this inequality, we use PARI/GP\footnote{\url{https://pari.math.u-bordeaux.fr/}} to compute $W_K$ and $N_{K^+/\mathbb{Q}}(\epsilon_{K^+})$ and compare them.  For the cases given in Table \ref{table, TZ}, the software was not able to compute $N_{K^+/\mathbb{Q}}(\epsilon_{K^+})$ for which we use some results of \cite{TZ}, as stated in the third column of Table \ref{table, TZ}, to compute $\# \Po(K)$. For all these cases, we conclude that $\#\Po(K)\neq h_K$.
\begin{table}
\begin{center}
	\begin{tabular}{|c |c |l |c|} 
		\hline
		$m$ & $n$ & $\# \Po(K)$ & $g_k=h_K$ \\ [0.5ex] 
		\hline\hline
		37,58 & 163 & 1 by \cite[Theorem 5]{TZ} & 2 \\ 
		\hline
		67 & 123,235,403 & 1 by \cite[Theorem 5]{TZ} & 2 \\ 
		\hline
		91 & 163 & 1 by \cite[Theorem 5]{TZ} & 2 \\ 
		\hline
		91 & 403 & 1 by \cite[Theorem 7]{TZ} & 2 \\ 
		\hline
		115 & 163 & 1 by \cite[Theorem 5]{TZ} & 2 \\ 
		\hline
		115 & 235 & 1 by \cite[Theorem 7]{TZ} & 2 \\ 
		\hline
		163 & 187,235,267,403 & 1 by \cite[Theorem 5]{TZ} & 2 \\ 
		\hline
		1 & 30, 42, 70, 78, 190 & 1 by \cite[Lemma 6]{TZ} & 4 \\ 
		\hline
		13 & 78 & 1 by \cite[Theorem 5]{TZ} & 4 \\ 
		\hline
	\end{tabular}
	\caption{The list of imaginary bi-quadratic fields  $K=\mathbb{Q}(\sqrt{-m},\sqrt{-n})$ with $g_K=h_K$ for which $\# \Po(K)$ is computed by using the results of \cite{TZ}.}
		 \label{table, TZ}
\end{center}
\end{table}

(iii) The list given in \cite[Table V]{CK} (or \cite[Table 3]{Miyada})
provides the complete classification of all imaginary tri-quadratic fields $K$ with $g_K=h_K$. We observe that $h_K=1$ for any field $K$ in this list. On the other hand, by  Proposition \ref{proposition, Po(K) and gK for K imaginary tri-quadratic}, the order of $\Po(K)$ divides the genus number $g_K$, hence $\# \Po(K) \leq g_K \leq h_K$. Consequently, Table \ref{tab, tri-quadratic Miyada} gives the list of all imaginary tri-quadratic fields with $\Po(K)=\Cl(K)$.

(iv) The complete classification of all imaginary non-quadratic cyclic fields $K$ with $g_K=h_K$ is given in \cite[Table 1]{CK}. Part (i) of Proposition \ref{proposition, Po(K) and gK for K cyclic}  shows that this is also the complete classification of  all imaginary non-quadratic cyclic fields whose P\'olya groups are equal to their class groups.
\end{proof}
\section{Proofs of other classification theorems}\label{five}

We start by expressing the order of the P\'olya group of a quadratic field in terms of the divisor function.

\begin{lemma} \label{lemma, order of Po is less than tau imaginary quadratic}
	Let $K=\mathbb{Q}(\sqrt{D})$ be a quadratic field, where
	$D$ is a square-free integer. Let
		\begin{equation*}
		c_K=
		\begin{cases} 
			1/2	, &   \textit{if} ~~   D \equiv 1 \, (\mathrm{mod}\, 4), \\
			1/3,  &   \textit{if} ~~ D \equiv 3 \, (\mathrm{mod}\, 4), \\
			1/4,   &  \textit{if} ~~ D \equiv 2 \, (\mathrm{mod}\, 4),\\
		\end{cases}
	\end{equation*}
	and denote by $\tau$ the number of divisors function. Then,
\begin{equation} \label{equation, Po is alpha times tau}
	\# \Po(K)= \begin{cases} 
	\frac{1}{2}	c_K \tau(d_K), &   \textit{if $D>0$ and $K$ has no units of negative norm}, \\
		c_K \tau(d_K),   &  \textit{otherwise}.\\
		\end{cases}
\end{equation}
\end{lemma}

\begin{proof}
	For  $n =\prod_{i=1}^k p_i^{\alpha_i}$ , we have 
		$\tau\left(n\right)= \prod_{i=1}^k(\alpha_i +1).$
	Thus,	
	\begin{equation*}
		\tau(d_K)=\begin{cases} 
			2^{s_K}, &    \textit{if} ~~ D \equiv 1 \, (\mathrm{mod}\, 4), \\
			3 \cdot 2^{s_K-1}, &    \textit{if} ~~ D \equiv 3 \, (\mathrm{mod}\, 4), \\
			2^{s_K+1}, & \,  \textit{if} ~~ D \equiv 2 \, (\mathrm{mod}\, 4),\\
		\end{cases}
	\end{equation*}
	where $s_K$ denotes the number of ramified primes in $K/\mathbb{Q}$. Now the assertion follows from \eqref{Hilbert}.
	\end{proof}

In the next lemma we establish an explicit upper bound for the divisor function.

\begin{lemma} \label{lemma, tau over n1/4}
	Let $n=2^{\beta} \prod_{i=1}^r q_i$ be a positive integer,  where $q_i$'s are distinct odd prime numbers and $\beta$ is a non-negative integer.
	Then 

	\begin{equation} \label{equation, Tau(n) over n1/4}
		\frac{\tau(n)}{\sqrt[4]{n}} \leq \begin{cases} 
			2.8908, &    \textit{if} ~~ \beta=0, \\
			\left(\frac{	\beta +1}{\sqrt[4]{2}}\right)\cdot 2.8908 , &    \textit{if} ~~ \beta \geq 1.
		\end{cases}
	\end{equation}
In particular, for a quadratic field $K=\mathbb{Q}(\sqrt{D})$, we have
	\begin{equation} \label{equation, tau over fourth root of dK 2}
		\tau(d_K) < c_K' \sqrt[4]{d_K},
	\end{equation}
	where
	\begin{equation*}
		c_K'=  \left\{
		\begin{array}{ll}
			2.8908, & \textit{if} ~~ D \equiv 1 \, (\mathrm{mod}\, 4),\\
			& \\
			7.2927, & \textit{if} ~~  D \equiv 3 \, (\mathrm{mod}\, 4),\\
			& \\
			9.7235,	& \textit{if} ~~ D \equiv 2 \, (\mathrm{mod}\, 4). \\
		\end{array}
		\right.
	\end{equation*}
	\end{lemma}

\begin{proof}
Since for each prime $p \geq 17$, we have $\frac{2}{\sqrt[4]{p}}<1$, then
	\begin{equation*}
		\frac{\tau(n)}{\sqrt[4]{n}} \leq  \left\{
		\begin{array}{ll}
			\displaystyle{\prod_{\substack{p \, \text{prime},\\ 2<p<17}}} \frac{2}{\sqrt[4]{p}}, &   \textit{if} ~~ \beta=0,\\
			& \\
			\frac{\beta +1}{\sqrt[4]{2^{\beta}}}	\displaystyle{\prod_{\substack{p \, \text{prime},\\ 2<p<17}}} \frac{2}{\sqrt[4]{p}}, &   \textit{if} ~~ \beta \geq 1.\\
		\end{array}
		\right.
	\end{equation*}
	The relation \eqref{equation, Tau(n) over n1/4} follows by noting that $2.8908$ is an upper bound for the product in the above inequality. The bound \eqref{equation, tau over fourth root of dK 2} is an immediate consequence of \eqref{equation, Tau(n) over n1/4}.
\end{proof}

\subsection{Proof of Theorem \ref{theorem, list imaginary [Cl(K):Po(K)]=2}} To prove Theorem \ref{theorem, list imaginary [Cl(K):Po(K)]=2}, we use Lemmas \ref{lemma, order of Po is less than tau imaginary quadratic}  and \ref{lemma, tau over n1/4} along with the following  conditional explicit lower bound, proved in \cite[Theorem 5]{Ihara06}, for the class numbers of imaginary quadratic fields. 

\begin{theorem} [{Ihara}] 
	\label{theorem, Ihara's upper bound}
	Let $K$ be an imaginary quadratic field with the absolute value discriminant $d_K$. Under GRH, if $\alpha_K:=\frac{1}{2} \log {d_K} > 1.16$ ( i.e., 
if $d_K \geq 11$), then
	\begin{equation} \label{equation, Ihara's upper bound}
		h_K > \frac{\frac{\pi}{6} \sqrt{{d_K}} -\alpha_K+b_1}{\alpha_K+2\log \alpha_K +b_2+c(\alpha_K)},
	\end{equation}
	with $b_1$, $b_2$, $c(\alpha_K)$ are given by
	\begin{equation*}
		\begin{cases} 
			b_1=2\log M+\log 2-4q_0=-0.34037-4q_0, \\
			b_2=2 \log M-2 \gamma_{\mathbb{Q}}+\log 2+1=-0.49480 \dots, &  \\
			c(t)=\frac{4 \log t+2}{t-1}, &  
		\end{cases}
	\end{equation*}
	where  $q_0=e^{-\pi \sqrt{{d_K}}}$, $M=0.596450134 \dots$, $\log M=-0.516759638 \dots$, and $\gamma_{\mathbb{Q}}$ is the Euler-Kronecker constant.
\end{theorem}

Now, we are ready to prove Theorem \ref{theorem, list imaginary [Cl(K):Po(K)]=2}.

\begin{proof}[Proof of Theorem \ref{theorem, list imaginary [Cl(K):Po(K)]=2}.]  
	By Lemmas \ref{lemma, order of Po is less than tau imaginary quadratic} and \ref{lemma, tau over n1/4} and Theorem \ref{theorem, Ihara's upper bound}, we have
	\begin{align} \label{equation, 2 Po(K) less than}
		\frac{2\cdot\#\Po(K)}{h_K} < \frac{c_K c_K' \sqrt[4]{d_K}\left( \frac{1}{2} \log {d_K} + 2 \log \log \sqrt{{d_K}} +b_2+\frac{4 \log \log \sqrt{{d_K}} +2}{\log \sqrt{{d_K}} -1} \right)}{\frac{\pi}{6} \sqrt{{d_K}}-\frac{1}{2} \log {d_K} +b_1},	
	\end{align}
	where 
 $b_1,b_2$ are the constants given in Theorem \ref{theorem, Ihara's upper bound}. 
 	
We define, for $K=\mathbb{Q}(\sqrt{-D})$ and $x>0$, the function $f_K(x)$ as
	\begin{equation*}
		f_K(x)=\frac{ c_K c_K' \sqrt[4]{x}\left(  \log \sqrt{x} + 2 \log \log \sqrt{x} +b_2+\frac{4 \log \log \sqrt{x} +2}{\log \sqrt{x} -1} \right)}{\frac{\pi}{6} \sqrt{x}- \log \sqrt{x} +b_1}.
	\end{equation*}
 We observe that
	\begin{equation*}
		\left\{
		\begin{array}{ll}
			f_K(x)<1, & \textit{if ~~ $x \geq 3.6\times 10^7$ and $-D \equiv 1 \, (\mathrm{mod}\, 4)$},\\
			& \\
			f_K(x) <1,	 & \textit{if ~~ $x \geq 4.1 \times 10^8$ and $-D \equiv 2,3 \, (\mathrm{mod}\, 4)$}.\\
		\end{array}
		\right.
	\end{equation*}
	Hence, by \eqref{equation, 2 Po(K) less than}, if $\left[\Cl(K): \Po(K)\right]=2$, under GRH, we must have 
	\begin{equation} \label{equation, upper bound over dk for 2Po(K)=Cl(K)}
		\left\{
		\begin{array}{ll}
			{d_K} < 3.6\times 10^7, & \textit{if} ~~ -D \equiv 1 \, (\mathrm{mod}\, 4),\\
			& \\
			{d_K} < 4.1  \times 10^8,  & \textit{if} ~~ -D \equiv 2,3 \, (\mathrm{mod}\, 4).\\
		\end{array}
		\right.
	\end{equation}
Using PARI/GP\footnote{\url{https://pari.math.u-bordeaux.fr/}} we compute $h_K$ and $\# \Po(K)$ for all imaginary quadratic fields $K$ with ${d_K}$ smaller than the bound \eqref{equation, upper bound over dk for 2Po(K)=Cl(K)} to obtain the classification given in Table \ref{tab, imaginiary quadratic t=2}.
\end{proof}

\subsection{Proof of Theorem \ref{theorem, finiteness Extended R-D type}} 
We need the following three lemmas. 
\begin{lemma}  \label{lemma, Degert}
If $K=\mathbb{Q}(\sqrt{\ell^2+r})$ is a real quadratic field of R-D type, then
	\begin{equation} \label{equation, f.u of R-D type} 
	u_K :=	\left\{
	\begin{array}{ll}
		\ell + \sqrt{\ell^2+r}, \, \, & \textit{if} ~~ |r|= 1,\\
		& \\
		\frac{	\ell + \sqrt{\ell^2+r}}{2}, \, \, & \textit{if} ~~ |r|= 4,\\
		& \\
		\frac{	2\ell^2+r + 2 \ell \sqrt{\ell^2+r}}{|r|}, \, \,  & \textit{if} ~~ |r|\neq 1,4,\\
	\end{array}
	\right.
\end{equation}
is the fundamental unit of $K$. More generally, $u_K$ is a unit if $K$ is of extended R-D type.
\end{lemma}

\begin{proof}
If $K$ is of R-D type, by \cite[Theorem 1]{Degert}, $u_K$ is the fundamental unit of $K$.
	Since $u_K$ has norm $\pm 1$, we only need to show that $u_K$ is an algebraic integer when $K$ is of extended R-D type. Let
	$D$ be an extended R-D type integer which is not of R-D type. Then $D=\ell^2+r$, where $r \mid 4 \ell$ and either $r > \ell$ or $r \leq -\ell$. 	If $|r|=4$, then for $r=4$ (resp. $r=-4$), $\ell \in \{1,2,3\}$ (resp. $\ell \in \{3,4 \}$.) It is straightforward to check that $u_K$ is an algebraic integer in these cases.
If $|r| \neq 4$, then	either $(\ell,r)=(1,2)$ or $\ell>1$ and $	r \in \{ -\ell, \pm 2 \ell_1, \pm 4 \ell_1 \}$ for some divisor $\ell_1$ of $\ell$. For $(\ell,r)=(1,2)$, we have $D=3$. In this case, $u_K=2+\sqrt{3}$ is an algebraic integer. Now let $\ell >1$. 
If $r=\pm 4 \ell_1$, then $D=\ell^2 \pm 4 \ell_1 \equiv 1 \, (\mathrm{mod}\, 4)$  (Note that in this case $\ell$ cannot be even, as $D$ is square-free.) Thus,
\begin{equation*}
		\frac{	2\ell^2+r + 2 \ell \sqrt{\ell^2+r}}{|r|}=\frac{(\ell^2/\ell_1 \pm 2)+(\ell^2/\ell_1)\sqrt{D} }{2}
\end{equation*}
 is an algebraic integer in $K$.
	Finally, if $r =-\ell,\pm 2 \ell_1$, then $r$ divides both $2 \ell^2+r$ and $2 \ell$ which implies that
	\begin{equation*}
\frac{	2\ell^2+r + 2 \ell \sqrt{\ell^2+r}}{|r|} \in \mathbb{Z}[\sqrt{D}].
	\end{equation*}\end{proof}

The following assertion is stated in \cite[p. 423]{MW}.

\begin{lemma}\label{lemma, log R_K less than log dK}
Let $K=\mathbb{Q}(\sqrt{D})$ be a real quadratic field of extended R-D type. Then $ R_K <\log (3 D)$, where $R_K$ denotes the regulator of $K$. 
 \end{lemma}

 \begin{proof} 	
 As usual, denote by $\epsilon_K$ the fundamental unit of $K$. First observe that for $D=5$, we have $\epsilon_K=\frac{-1+\sqrt{5}}{2}$. So, $	R_K=\log\left( \epsilon_K \right) \approx 0.482 < \log(3 D) \approx 2.7$.

Next, let $D=\ell^2+r \neq 5$ with $r \mid 4\ell$.  
We have three cases:

\noindent 	
\textit{Case 1.} If $|r|=1$, then by Lemma \ref{lemma, Degert}, 
 		\begin{equation*}
 			\epsilon_K =\ell + \sqrt{\ell^2 \pm 1} < \left(2\ell^2 \pm 2\right)+\left(\ell^2 \pm 1\right)=3D.
 		\end{equation*}
 (Note that $\ell \geq 2$ for $r=-1$.)
 	 
\noindent
\textit{Case 2.} If $|r|=4$, then, by Lemma \ref{lemma, Degert},
	\begin{equation*}
		u_K=\frac{	\ell + \sqrt{\ell^2+r}}{2}
	\end{equation*}
	is a unit of $K$. 
		Hence, 
	\begin{equation*}
		\epsilon_K \leq 	\frac{	\ell + \sqrt{\ell^2 \pm 4}}{2} 	< \left(2 \ell^2\pm 8\right) + \left(\ell^2\pm 4\right)=3D.
	\end{equation*}
	(Note that $\ell \geq 3$ for $r=-4$.)

\noindent 
\textit{Case 3.} if $|r| \neq 1,4$, then by Lemma \ref{lemma, Degert}, 
	\begin{equation*}
		u_K=\frac{	2\ell^2+r + 2 \ell \sqrt{\ell^2+r}}{|r|}
	\end{equation*}
 	is a unit of $K$. If $\ell=1$, the only possibility is to have $r=2$, i.e., $D=3$. In this case, $\epsilon_K=2+\sqrt{3}$ and $	R_K=\log{\epsilon_{K}} \approx 1.31 < \log (3D) \approx 2.2$.

Now assume that $\ell >1$. Since $r \leq 4 \ell \leq 2 \ell^2$, then $\ell \sqrt{\ell^2+r} <  2 \ell^2$.
 
 \noindent
For $r>0$,
 		\begin{equation*}
 			\epsilon_K \leq 	\frac{	2\ell^2+r + 2 \ell \sqrt{\ell^2+r}}{|r|} \leq 	\frac{	2\ell^2+r + 2 \ell \sqrt{\ell^2+r}}{2}  < \left(	\ell^2+ \frac{r}{2} \right) + 2 \ell^2 < 3 \ell^2 + 3r=3D.
 		\end{equation*}
 		
  \noindent
 	For $r <0$, we have
 			\begin{equation} \label{equation, epsilon less l^2}
 			\epsilon_K \leq 	\frac{	2\ell^2+r + 2 \ell \sqrt{\ell^2+r}}{|r|} <	\frac{	2\ell^2 + 2 \ell \sqrt{\ell^2}}{2}= 2\ell^2.
 		\end{equation}
 	Since $r \mid 4 \ell$, for $\ell \geq 12$, we get
 			$-3 r \leq 12 \ell \leq \ell^2.$
 		Consequently, from \eqref{equation, epsilon less l^2} we find that for $\ell \geq 12$, the inequality $\epsilon_K < 2 \ell^2 \leq 3 \ell^2+3r=3D$ holds.
 	We numerically check that the relation
 	\begin{equation*}
 		\frac{	2\ell^2+r + 2 \ell \sqrt{\ell^2+r}}{|r|} < 3 \ell^2+3r=3D
 	\end{equation*}
 	also holds for $1 \leq \ell \leq 11$ and $r \mid 4 \ell$ with $r \neq \pm 1, \pm 4$.
 \end{proof}

The following explicit lower bound for the class numbers of quadratic fields of extended R-D type is a consequence of the analytic class number formula, Lemma \ref{lemma, log R_K less than log dK}, and Tatuzawa's lower bound \cite{Tatuzawa} for the values of Dirichlet $L$-functions at $1$. 

\begin{lemma} \label{lemma, lower bound RD}
	Let $K=\mathbb{Q}(\sqrt{D})$ run through the family of real quadratic fields of extended R-D type. Then, with only one possible exception, we have
	\begin{equation} \label{equation, lower bound for hK of E-R-D type}
		h_K > \frac{0.655 \, d_K^{7/16}}{32 \log (3D)}.
	\end{equation}

\end{lemma}

\begin{proof}
See \cite[p. 423]{MW} for a proof.
\end{proof}

We are now ready to prove Theorem \ref{theorem, finiteness Extended R-D type}.
\begin{proof}[Proof of Theorem \ref{theorem, finiteness Extended R-D type}.] Replacing \eqref{equation, Ihara's upper bound} with \eqref{equation, lower bound for hK of E-R-D type}, we follow the same method as in the proof of Theorem \ref{theorem, list imaginary [Cl(K):Po(K)]=2}. 
For $K=\mathbb{Q}(\sqrt{D})$, we  consider two cases:

\noindent
 \textit{Case 1.} 
		If $N_{K/\mathbb{Q}}(\epsilon_K)=-1$, then by Lemmas \ref{lemma, order of Po is less than tau imaginary quadratic}, \ref{lemma, tau over n1/4}, and \ref{lemma, lower bound RD}, we have
		\begin{equation} \label{equation, Po(K) over hK real E-R-D}
			\frac{\# \Po(K)}{h_K} < \frac{32 c_K c_K' \log (3D)}{0.655 \, d_K^{3/16}} \leq \frac{32 c_K c_K' \log (3d_K)}{0.655 \, d_K^{3/16}}. 
		\end{equation}
Observe that for the function
		\begin{equation*}
			f_{K}(x):=\frac{32 c_K c_K' \log 3x}{0.655 \, x^{3/16}}, \quad x>0,
		\end{equation*}
	we have
		\begin{equation*}
			\left\{
			\begin{array}{ll}
				f_K(x)<1, & \textit{if ~~ $x \geq 4.3 \times 10^{18}$ and  $D \equiv 1 \, (\mathrm{mod}\, 4)$},\\
				& \\
				f_K(x) <1,	 & \textit{if ~~ $x \geq 8.14 \times 10^{19}$ and   $D \equiv 2,3 \,  (\mathrm{mod}\, 4)$}.\\
			\end{array}
			\right.
		\end{equation*}
		
		Consequently, for a real quadratic field $K=\mathbb{Q}(\sqrt{D})$ of extended R-D type, with one possible exception,  if $K$ has some units of negative norm and $\Po(K)=\Cl(K)$, then
		\begin{equation}  \label{equation, upper bound for dK and f.u=-1}
			\left\{
			\begin{array}{ll}
				d_K < 4.3\times 10^{18}, & \textit{if} ~~ D \equiv 1 \, (\mathrm{mod}\, 4),\\
				& \\
				d_K < 8.14 \times 10^{19},  & \textit{if} ~~ D \equiv 2,3 \, (\mathrm{mod}\, 4).\\
			\end{array}
			\right.
		\end{equation}

		\noindent
	\textit{Case 2.} Let $K$ has no units of negative norm, i.e., $N_{K/\mathbb{Q}}(\epsilon_K)=+1$. In this case,  
	 by an argument similar to Case 1, we find that if $\Po(K)=\Cl(K)$, then
		\begin{equation} \label{equation, upper bound for dK and f.u=+1}
			\left\{
			\begin{array}{ll}
				d_K < 6.3\times 10^{16}, & \textit{if} ~~ D \equiv 1 \, (\mathrm{mod}\, 4),\\
				& \\
				d_K < 1.3 \times 10^{18},  & \textit{if} ~~ D \equiv 2,3 \, (\mathrm{mod}\, 4).\\
			\end{array}
			\right.
		\end{equation}

Using PARI/GP\footnote{\url{https://pari.math.u-bordeaux.fr/} \label{1}}, we compute $h_K$ and $\# \Po(K)$ for all real quadratic fields $K$ of extended R-D type for which the discriminant $d_K$ satisfies  in \eqref{equation, upper bound for dK and f.u=-1} or \eqref{equation, upper bound for dK and f.u=+1}  to get Table \ref{tab,E.R-D}. 
\end{proof}

\subsection{Proof of Theorem \ref{theorem, E-R-D real quadratic gK=hK}.}
\begin{proof}[Proof of Theorem \ref{theorem, E-R-D real quadratic gK=hK}.]
Let $K=\mathbb{Q}(\sqrt{D})$ be a real quadratic field of extended R-D type such that $g_K=h_K$. We consider two cases:

\noindent
\emph{Case 1.} If $\#\Po(K)=h_K$, i.e., if $K$ is any of those fields given in Table \ref{tab,E.R-D}, then $	\# \Po(K)=g_K=h_K$. This is true, since $g_K \leq h_K$ and by Corollary \ref{corollary, Po(K),gK and gK+ real quadratic}, $\# \Po(K) \mid g_K$.

\noindent
\emph{Case 2.} If $\# \Po(K) \neq h_K$, then from $g_K=h_K$ and   Corollary \ref{corollary, Po(K),gK and gK+ real quadratic}, we get 
\begin{equation} \label{equation, Polya index 2 E-R-D}
	\# \Po(K)=g_K/2=h_K/2.
\end{equation}
Moreover, $K$ has no units of negative norm and $D$ has no prime divisor $p \equiv 3 \, (\mathrm{mod}\, 4)$. 
Now, by an argument similar to the proof of Theorem \ref{theorem, finiteness Extended R-D type}, we deduce that in this case, if \eqref{equation, Polya index 2 E-R-D} holds, then 
	\begin{equation}  \label{equation, upper bound on dK for E-R-D eith gk=hk}
	\left\{
	\begin{array}{ll}
		d_K < 4.3\times 10^{18}, & \textit{if} ~~ D \equiv 1 \, (\mathrm{mod}\, 4),\\
		& \\
		d_K < 8.14 \times 10^{19},  & \textit{if} ~~ D \equiv 2,3 \, (\mathrm{mod}\, 4).\\
	\end{array}
	\right.
\end{equation}

 We use PARI/GP\footref{1} to list all real quadratic fields $K=\mathbb{Q}(\sqrt{D})$ of extended R-D type  with $h_K=2^{s_K-1}$ such that $K$ has no units of negative norm, $D$ has no prime divisor $p \equiv 3 \, (\mathrm{mod}\, 4)$, and the discriminant $d_K$ is bounded by \eqref{equation, upper bound on dK for E-R-D eith gk=hk} (Note that by \eqref{Hilbert}, $\#\Po(K)=2^{s_K-2}$, and by \eqref{equation, Polya index 2 E-R-D}, we get $g_K=h_K=2^{s_K-1}$.) Using these computations and Corollary \ref{corollary, Po(K),gK and gK+ real quadratic}, we get the list, given  in Table \ref{tab,E.R-D2}, of all real quadratic fields $K=\mathbb{Q}(\sqrt{D})$ of extended R-D type, with $D < 4.3 \times 10^{18}$, such that $g_K=h_K$ and $g_K^+ \neq h_K^+$.
\end{proof}

\newpage

\section{Tables}\label{six}

{\small

\begin{table}[!h]
	\begin{center}
		\begin{tabular}{|c | c|p{6.5cm}|} 
			\hline
			$\#\Po(K)=h_K$ & $m$ & $n$  \\ [0.5ex] 
			\hline \hline
			\multirow{8}{0.6em}{$1$} & $1$ & 
			$2,3,5,7,11,13,19,37,43,67,163$ \\ 	\cline{2-3} 
			& $2$ & $3,7,10,11,19,37,43,58,67$ \\  \cline{2-3} 
			& $3$ & $6,7,11,15,19,43,51,67,123,163,267$ \\ \cline{2-3} 
			& $7$ & $11,19,35,43,91,163,247$\\ \cline{2-3} 
			& $11$ &  $19,22,67.163,187$ \\ \cline{2-3} 
			& $19$ & $67,163$ \\ \cline{2-3} 
			& $43$& $67,163$ \\ \cline{2-3} 
			& $67$ & $163$ \\  
			\hline 
			\multirow{3}{0.6em}{$2$} & $1$ & $6,10,58$ \\ \cline{2-3} 
			& $2$ & $5,6,13,22,37$ \\ \cline{2-3} 
			& $5$ & $10$ \\  \cline{2-3} 
			\hline
		\end{tabular} \caption{The list of all imaginary bi-quadratic fields $K=\mathbb{Q}(\sqrt{-m},\sqrt{-n})$ with $\Po(K)=\Cl(K)$.} \label{tab, bi-quadratic Po=h}
	\end{center}
\end{table}

\begin{table}[!h]
	\begin{center}
		\begin{tabular}{|c|p{7.5cm}|} 
			\hline
			$\#\Po(K)=h_K$ & $( m_1,  m_2, m_3)$   \\ [0.5ex] 
			\hline \hline
			\multirow{4}{0.6em}{$1$} & $(1,2,5),(1,2,3),(1,3,5),(1,2,11),(3,5,7),$ \\
			& $(3,5,2),(1,5,7),(3,2,11),(7,5,2), (1,13,7),$ \\
			& $(3,17,11),(1,3,7),(1,3,11),(3,7,2),$ \\
			& $(1,3,19),(1,7,19),(3,11,19)$
			\\ \hline
		\end{tabular} 
		\caption{
			The list of all imaginary tri-quadratic fields\\
		\hspace*{1.9cm} $K=\mathbb{Q}(\sqrt{- m_1},\sqrt{- m_2}, \sqrt{- m_3})$ with $\Po(K)=\Cl(K)$.} 
	\label{tab, tri-quadratic Miyada}
\end{center}
\end{table}

	\begin{table}[!h]
	\begin{center}
		\begin{tabular}{|c|p{10.2cm}|} 
			\hline
			$h_K=2\cdot\#\Po(K)$ & $D$   \\ [0.5ex] 
			\hline \hline
			\multirow{3}{0.6em}{$4$} & 
			$14,17,34,39,46,55,73,82,97,142,155,193,203,219,259,$ \\ 
			& $291,323,355,667,723,763,955,1003,1027,1227,1243,$ \\ 
			& $1387,1411,1507,1555$ \\ 
			\hline 
			\multirow{5}{0.6em}{$8$} & 
			$65,66,69,77,114,138,141,145,154,205,213,217,238,258,$ \\ 
			& $265,282,301,310,322,418,438,442,445,498,505,553,598,$ \\ 
			& $651,658,697,742,793,915,987,1131,1443,1635,1659,$ \\ 
			& $1771,1947,2035,2067,2139,2163,2451,2667,2715,2755,$ \\
			& $3243,3355,3507,4123,4323,5083,5467,6307$ \\
			\hline 
			\multirow{6}{0.6em}{$16$} & 
			$285,390,429,465,510,561,570,609,645,690,777,798,805,$ \\
			& $858,870,897,910,957,1005,1045,1065,1105,1110,1113,$ \\
			& $1122,1185,1290,1302,1353,1605,1645,1653,1677,1705,$ \\
			& $1870,1885,2002,2013,2170,2233,2737,3795,4515,5115,$ \\
			& $5187,6195,7035,7315,7395,7755,7995,8547,8715,8835,$ \\
			& $9867,11067,11715,13195,14763,16555$ \\
			\hline 
			\multirow{2}{0.6em}{$32$} & 
			$1785,2145,2310,2730,3045,3570,3705,4305,4830,4845,$	\\ 
			& $5005,19635,31395,33915,40755$ \\
			\hline 
		\end{tabular} \caption{The list of all the imaginary quadratic fields $K=\mathbb{Q}(\sqrt{-D})$ with $\left[\Cl(K):\Po(K)\right]=2$ (Under GRH).}
\label{tab, imaginiary quadratic t=2}
	\end{center}
\end{table}

\begin{table}[!h]
	\begin{center}
		\begin{tabular}{|c|c|p{9.2cm}|} 
			\hline
			$g_K=\#\Po(K)=h_K$ & $g_K^+=h_K^+$ & $D$   \\ [0.5ex] 
			\hline \hline
			$1$ & $1$ & $2, 5, 13, 17, 29, 37, 53, 101, 173, 197, 293, 677$ \\
			\hline
			\multirow{3}{0.6em}{$1$}  &\multirow{3}{0.6em}{$2$} & $3, 6, 7, 11, 14, 21, 23, 33, 38, 47, 62, 69, 77, 83, 93, 141,$ \\
			& & $167, 213, 227,237, 398, 413, 437, 453, 573, 717, 1077,$ \\
			& & $1133, 1253, 1293, 1757$ \\
			\hline 
			\multirow{2}{0.6em}{$2$}& 	\multirow{2}{0.6em}{$2$} & $10, 26, 65, 85, 122, 362, 365, 485, 533, 629, 965, 1157,$ \\
			& & $1685, 1853, 2117, 2813, 3365$ \\
			\hline 
			\multirow{7}{0.6em}{$2$}& 	\multirow{7}{0.6em}{$4$} &  $15, 30, 35, 39, 42, 51, 66, 78, 87, 95, 102,105, 110, 119,$ \\
			& & $ 123, 138, 143, 165, 182, 203, 215, 222, 230, 258, 285,287,$  \\
			& & $  318, 327, 357, 395, 402, 429, 447, 527, 623, 635, 645, 678, $ \\
			& & $ 741,782, 843, 885, 902, 957, 1022, 1085, 1173,1245, 1298,$ \\
			& & $  1533, 1605, 1965, 2013,2037, 2085, 2093, 2301, 2373,$ \\
			&& $ 2397,2613, 2717, 3237,3597, 3605, 3813,  4245, 4277, $ \\
			& & $4773, 4893, 5757, 5885,5957, 6573, 7733, 14405$ \\
			\hline 
			\multirow{2}{0.6em}{$4$}& 	\multirow{2}{0.6em}{$4$} & $170, 290, 530, 962, 1370, 2405, 3485, 9605,10205, 14885,$ \\
			& & $ 16133, 20165$ \\
			\hline	
			\multirow{10}{0.6em}{$4$}& 	\multirow{10}{0.6em}{$8$} & $195, 210, 231, 255, 330, 390, 435, 455, 462, 483,570, 615,$ \\ 
			& & $627, 663, 770, 798, 903, 915, 930, 1095, 1190, 1218, 1230,$ \\
			& & $1235, 1295, 1302, 1365, 1463, 1482, 1515, 1518, 1547, 1595,$ \\
			& & $1610,1722, 1767, 1770, 1938, 2015, 2030, 2387, 2595, 2607,$ \\
			& & $ 2618, 2805, 2910, 3045, 3230, 3335, 3723, 3885, 4389, 4485,$ \\
			& & $4758, 5565, 6045, 6405, 6765, 7293, 7917, 8645, 9933, $ \\
			& & $10005,10965, 11165, 12045, 13485, 13845, 14685, 15645, $ \\
			& & $16653, 17765,19565, 20405, 21045, 24045, 30597, 31317,$ \\
			& & $33117, 41613$ \\
			\hline 
			$8$ & $8$ & $2210, 5330, 32045, 58565, 77285$ \\ \hline 
			\multirow{4}{0.6em}{$8$}& 	\multirow{4}{0.6em}{$16$} & $11155, 1995, 2310, 2415, 2730, 3003, 3135, 3255, 3570,$ \\
			& & $3927, 3990, 4290, 4935, 5187, 5610, 5655, 6090, 6555, 7035,$ \\ 
			& & $7215, 7755, 9030, 10010, 12155, 12558, 13695, 14630,  $ \\
			& & $21318,23205, 26565, 35805, 74613, 108885$ \\ \hline 
			$16$ & $32$ & $19635, 25935, 33495, 451605$ \\
			\hline
		\end{tabular} \caption{The list of all  real quadratic fields $K=\mathbb{Q}(\sqrt{D})$ of extended R-D type, with  $D < 6.3 \times 10^{16}$, for which $g_K=\#\Po(K)=h_K$.
			This is  also the list of all such fields $K$ with $g_K^+=h_K^+$.} \label{tab,E.R-D}
	\end{center}
\end{table}

	\begin{table}[!h]
	\begin{center}
		\begin{tabular}{|c|p{10.4cm}|} 
			\hline
			$g_K=2\cdot\# \Po(K)=h_K$ & $D$   \\ [0.5ex] 
			\hline \hline
				\multirow{2}{0.5em}{$2$} & $34, 146, 194, 205, 221, 482, 1205, 1469, 1517, 2045, 3005, 5645,$ \\ 
				& $7157, 8333, 9005$ \\ 
			\hline 
			$4$ & $410, 890, 3842, 3965, 7565, 7685, 18245, 25493, 41093, 55205$ \\
			\hline
			$8$ & $10370, 19610, 22490$ \\
			\hline
			$16$ & $81770$ \\
			\hline
		\end{tabular} \caption{The list of all real quadratic fields $K=\mathbb{Q}(\sqrt{D})$ of extended R-D type, with $D < 4.3 \times 10^{18}$, such that $g_K=h_K$ and $g_K^+ \neq h_K^+$.
		 } \label{tab,E.R-D2}
	\end{center}
\end{table}
}

\clearpage

\begin{rezabib}

\bib{Arno}{article}{
   author={Arno, Steven},
   title={The imaginary quadratic fields of class number $4$},
   journal={Acta Arith.},
   volume={60},
   date={1992},
   number={4},
   pages={321--334},
   issn={0065-1036},
   review={\MR{1159349}},
   doi={10.4064/aa-60-4-321-334},
}


	\bib{Cahen-Chabert's book}{book}{
	author={Cahen, Paul-Jean},
	author={Chabert, Jean-Luc},
	title={Integer-valued polynomials},
	series={Mathematical Surveys and Monographs},
	volume={48},
	publisher={American Mathematical Society, Providence, RI},
	date={1997},
	pages={xx+322},
	isbn={0-8218-0388-3},
	review={\MR{1421321}},
	doi={10.1090/surv/048},
}

	\bib{ChabertII}{article}{
	author={Chabert, Jean-Luc},
	author={Halberstadt, Emmanuel},
	title={From P\'olya fields to P\'olya groups (II): Non-Galois number
		fields},
	journal={J. Number Theory},
	volume={220},
	date={2021},
	pages={295--319},
	issn={0022-314X},
	review={\MR{4177545}},
	doi={10.1016/j.jnt.2020.06.008},
}

	\bib{CK}{article}{
	author={Chang, Ku-Young},
	author={Kwon, Soun-Hi},
	title={The imaginary abelian number fields with class numbers equal to
		their genus class numbers},
	language={English, with English and French summaries},
	note={Colloque International de Th\'eorie des Nombres (Talence, 1999)},
	journal={J. Th\'eor. Nombres Bordeaux},
	volume={12},
	date={2000},
	number={2},
	pages={349--365},
	issn={1246-7405},
	review={\MR{1823189}},
	doi={10.5802/jtnb.283},
}

\bib{NChildress}{book}{
	author={Childress, Nancy},
	title={Class field theory},
	series={Universitext},
	publisher={Springer, New York},
	date={2009},
	pages={x+226},
	isbn={978-0-387-72489-8},
	review={\MR{2462595}},
	doi={10.1007/978-0-387-72490-4},
}

\bib{Chowla}{article}{
	title={An extension of Heilbronn's class-number theorem},
	author={Chowla, S.},
	journal={The Quarterly Journal of Mathematics},
	volume={1},
	pages={304--307},
	year={1934},
	publisher={Oxford University Press},
}

\bib{Degert}{article}{
	author={Degert, G\"unter},
	title={\"Uber die Bestimmung der Grundeinheit gewisser
		reell-quadratischer Zahlk\"orper},
	language={German},
	journal={Abh. Math. Sem. Univ. Hamburg},
	volume={22},
	date={1958},
	pages={92--97},
	issn={0025-5858},
	review={\MR{0092824}},
	doi={10.1007/BF02941943},
}

\bib{Dickson}{book}{
title={Introduction to the Theory of Numbers},
author={Dickson, Leonard Eugene},
year={1929},
publisher={University of Chicago Press},
}

\bib{Kazuhiro}{article}{
	author={Dohmae, Kazuhiro},
	title={On real quadratic fields with a single class in each genus},
	journal={Japan. J. Math. (N.S.)},
	volume={19},
	date={1993},
	number={2},
	pages={241--250},
	issn={0289-2316},
	review={\MR{1265653}},
	doi={10.4099/math1924.19.241},
}

\bib{Emmelin}{article}{
title={A note on P\'{o}lya groups},
author={Emmelin, {\'E}tienne},
journal={arXiv:2302.07977, \url{https://arxiv.org/abs/2302.07977}},
year={2023},
pages={11 pages},
}

\bib{Gauss}{book}{
	author={Gauss, Carl Friedrich},
	title={Disquisitiones arithmeticae},
	note={Translated into English by Arthur A. Clarke, S. J},
	publisher={Yale University Press, New Haven, Conn.-London},
	date={1966},
	pages={xx+472},
	review={\MR{0197380}},
}

\bib{Hardy}{book}{
  	author={Hardy, G. H.},
  	author={Wright, E. M.},
  	title={An introduction to the theory of numbers},
  	edition={6},
  	note={Revised by D. R. Heath-Brown and J. H. Silverman;
  		With a foreword by Andrew Wiles},
  	publisher={Oxford University Press, Oxford},
  	date={2008},
  	pages={xxii+621},
  	isbn={978-0-19-921986-5},
  	review={\MR{2445243}},
  }

\bib{Heilbronn}{article}{
	title={On the class-number in imaginary quadratic fields},
	author={Heilbronn, Hans},
	journal={The Quarterly Journal of Mathematics},
	number={1},
	pages={150--160},
	year={1934},
	publisher={Oxford University Press},
}

\bib{Hirabayashi}{article}{
	author={Hirabayashi, Mikihito},
	author={Yoshino, Ken-ichi},
	title={Unit indices of imaginary abelian number fields of type $(2,2,2)$},
	journal={J. Number Theory},
	volume={34},
	date={1990},
	number={3},
	pages={346--361},
	issn={0022-314X},
	review={\MR{1049510}},
	doi={10.1016/0022-314X(90)90141-D},
}

\bib{Ihara06}{article}{
 	author={Ihara, Yasutaka},
 	title={On the Euler-Kronecker constants of global fields and primes with
 		small norms},
 	conference={
 		title={Algebraic geometry and number theory},
 	},
 	book={
 		series={Progr. Math.},
 		volume={253},
 		publisher={Birkh\"auser Boston, Boston, MA},
 	},
 	isbn={978-0-8176-4471-0},
 	isbn={0-8176-4471-7},
 	date={2006},
 	pages={407--451},
 	review={\MR{2263195}},
 	doi={10.1007/978-0-8176-4532-8\_5},
 }

	\bib{Ishida}{book}{
  	author={Ishida, Makoto},
  	title={The genus fields of algebraic number fields},
  	series={Lecture Notes in Mathematics},
  	volume={Vol. 555},
  	publisher={Springer-Verlag, Berlin-New York},
  	date={1976},
  	pages={vi+116},
  	review={\MR{0435028}},
  }

\bib{Kani}{article}{
	author={Kani, Ernst},
	title={Idoneal numbers and some generalizations},
	language={English, with English and French summaries},
	journal={Ann. Sci. Math. Qu\'ebec},
	volume={35},
	date={2011},
	number={2},
	pages={197--227},
	issn={0707-9109},
	review={\MR{2917832}},
}

\bib{Kubota}{article}{
		author={Kubota, Tomio},
		title={\"Uber den bizyklischen biquadratischen Zahlk\"orper},
		language={German},
		journal={Nagoya Math. J.},
		volume={10},
		date={1956},
		pages={65--85},
		issn={0027-7630},
		review={\MR{0083009}},
	}

\bib{Lang}{book}{
   	author={Lang, Serge},
   	title={Algebraic number theory},
   	publisher={Addison-Wesley Publishing Co., Inc., Reading, Mass.-London-Don
   		Mills, Ont.},
   	date={1970},
   	pages={xi+354},
   	review={\MR{0282947}},
   }

\bib{Lemmermeyer}{book}{
   author={Lemmermeyer, Franz},
   title={Quadratic number fields},
   series={Springer Undergraduate Mathematics Series},
   note={Translated from the 2017 German original},
   publisher={Springer, Cham},
   date={[2021] \copyright 2021},
   pages={xi+343},
   isbn={978-3-030-78651-9},
   isbn={978-3-030-78652-6},
  review={\MR{4331428}},
   doi={10.1007/978-3-030-78652-6},
}

\bib{Leopoldt}{article}{
   	author={Leopoldt, Heinrich W.},
   	title={Zur Geschlechtertheorie in abelschen Zahlk\"orpern},
   	language={German},
   	journal={Math. Nachr.},
   	volume={9},
   	date={1953},
   	pages={351--362},
   	issn={0025-584X},
   	review={\MR{0056032}},
   	doi={10.1002/mana.19530090604},
   }

\bib{Leriche 2014}{article}{
	author={Leriche, Amandine},
	title={About the embedding of a number field in a P\'olya field},
	journal={J. Number Theory},
	volume={145},
	date={2014},
	pages={210--229},
	issn={0022-314X},
	review={\MR{3253301}},
	doi={10.1016/j.jnt.2014.05.002},
}

\bib{Louboutin}{article}{
   	author={Louboutin, St\'ephane},
   	title={A finiteness theorem for imaginary abelian number fields},
   	journal={Manuscripta Math.},
   	volume={91},
   	date={1996},
   	number={3},
   	pages={343--352},
   	issn={0025-2611},
   	review={\MR{1416716}},
   	doi={10.1007/BF02567959},
   }

\bib{Malle}{article}{
	author={Malle, Gunter},
	title={On the distribution of Galois groups},
	journal={J. Number Theory},
	volume={92},
	date={2002},
	number={2},
	pages={315--329},
	issn={0022-314X},
	review={\MR{1884706}},
	doi={10.1006/jnth.2001.2713},
}

\bib{Abbas}{article}{
	author={Maarefparvar, Abbas},
	title={On P\'olya groups of non-Galois number fields},
	 journal={ To appear in Functiones et Approximatio, Commentarii Mathematici. arXiv: 2408.09019 \url{https://www.arxiv.org/abs/2408.09019}},
	 	 date={2025},
	 pages={20 pages},
	}

\bib{MW}{article}{
  	author={Mollin, R. A.},
  	author={Williams, H. C.},
  	title={Solution of the class number one problem for real quadratic fields
  		of extended Richaud-Degert type (with one possible exception)},
  	conference={
  		title={Number theory},
  		address={Banff, AB},
  		date={1988},
  	},
  	book={
  		publisher={de Gruyter, Berlin},
  	},
  	isbn={3-11-011723-1},
  	date={1990},
  	pages={417--425},
  	review={\MR{1106676}},
  }

\bib{milneCFT}{misc}{
author={ J. S. Milne},
title={Class Field Theory (v4.03)},
year={2020},
note={Available at www.jmilne.org/math/},
pages={287+viii}
}

\bib{Miyada}{article}{
  	author={Miyada, Ichiro},
  	title={On imaginary abelian number fields of type $(2,2,\cdots,2)$ with
  		one class in each genus},
  	journal={Manuscripta Math.},
  	volume={88},
  	date={1995},
  	number={4},
  	pages={535--540},
  	issn={0025-2611},
  	review={\MR{1362937}},
  	doi={10.1007/BF02567840},
  }

\bib{Murty}{article}{
   author={Murty, V. Kumar},
   title={Class numbers of CM-fields with solvable normal closure},
   journal={Compositio Math.},
   volume={127},
   date={2001},
   number={3},
   pages={273--287},
   issn={0010-437X},
   review={\MR{1845038}},
   doi={10.1023/A:1017589432526},
}

\bib{Nagel}{article}{
   author={Nagel, Trygve},
   title={Zur Arithmetik der Polynome},
   language={German},
   journal={Abh. Math. Sem. Univ. Hamburg},
   volume={1},
   date={1922},
   number={1},
   pages={178--193},
   issn={0025-5858},
   review={\MR{3069398}},
   doi={10.1007/BF02940590},
}

\bib{Neukirch}{book}{
	author={Neukirch, J\"urgen},
	author={Schmidt, Alexander},
	author={Wingberg, Kay},
	title={Cohomology of number fields},
	series={Grundlehren der mathematischen Wissenschaften [Fundamental
		Principles of Mathematical Sciences]},
	volume={323},
	edition={2},
	publisher={Springer-Verlag, Berlin},
	date={2008},
	pages={xvi+825},
	isbn={978-3-540-37888-4},
	review={\MR{2392026}},
	doi={10.1007/978-3-540-37889-1},
}

\bib{Peng-Jie}{article}{
	author={Wong, Peng-Jie},
	title={On Stark's class number conjecture and the generalised
		Brauer-Siegel conjecture},
	journal={Bull. Aust. Math. Soc.},
	volume={106},
	date={2022},
	number={2},
	pages={288--300},
	issn={0004-9727},
	review={\MR{4476085}},
	doi={10.1017/S0004972721001076},
}

	\bib{TZ}{article}{
  	author={Taous, Mohammed},
  	author={Zekhnini, Abdelkader},
  	title={P\'olya groups of the imaginary bicyclic bi-quadratic number
  		fields},
  	journal={J. Number Theory},
  	volume={177},
  	date={2017},
  	pages={307--327},
  	issn={0022-314X},
  	review={\MR{3629246}},
  	doi={10.1016/j.jnt.2017.01.025},
  }

\bib{Tatuzawa}{article}{
	title={On a theorem of Siegel},
	author={Tatuzawa, Tikao},
	booktitle={Japanese journal of mathematics: transactions and abstracts},
	volume={21},
	pages={163--178},
	year={1952},
	organization={The Mathematical Society of Japan}}

	\bib{Bennett}{article}{
  	author={Setzer, C. Bennett},
  	title={Units over totally real $C\sb{2}\times C\sb{2}$\ fields},
  	journal={J. Number Theory},
  	volume={12},
  	date={1980},
  	number={2},
  	pages={160--175},
  	issn={0022-314X},
  	review={\MR{0578808}},
  	doi={10.1016/0022-314X(80)90049-9},
  }

	\bib{Weinberger}{article}{
	author={Weinberger, P. J.},
	title={Exponents of the class groups of complex quadratic fields},
	journal={Acta Arith.},
	volume={22},
	date={1973},
	pages={117--124},
	issn={0065-1036},
	review={\MR{0313221}},
	doi={10.4064/aa-22-2-117-124},
}

	\bib{Zantema}{article}{
  	author={Zantema, H.},
  	title={Integer valued polynomials over a number field},
  	journal={Manuscripta Math.},
  	volume={40},
  	date={1982}, 
  	number={2-3},
  	pages={155--203},
  	issn={0025-2611},
  	review={\MR{0683038}},
  	doi={10.1007/BF01174875},
  }

\end{rezabib}

\end{document}